\documentclass[reqno]{amsart}
\usepackage[margin=1.5in,bottom=1.5in]{geometry}		

\usepackage{amsmath}		
\usepackage{amssymb}		
\usepackage{amsfonts}		
\usepackage{amsthm}		
\usepackage[foot]{amsaddr}		

\usepackage{mathtools}		

\mathtoolsset{%
}

\usepackage[utf8]{inputenc}		
\usepackage[T1]{fontenc}		

\usepackage[
cal=cm,
]
{mathalfa}

\usepackage{dsfont}		

\usepackage[sf,mono=false]{libertine}

\usepackage{acronym}		
\renewcommand*{\aclabelfont}[1]{\acsfont{#1}}		
\newcommand{\acli}[1]{\emph{\acl{#1}}}		
\newcommand{\acdef}[1]{\emph{\acl{#1}} \textup{(\acs{#1})}\acused{#1}}		
\newcommand{\acdefp}[1]{\emph{\aclp{#1}} \textup{(\acsp{#1})}\acused{#1}}	

\usepackage[labelfont={bf,small},labelsep=colon,font=small]{caption}	
\usepackage{subcaption}		
\usepackage[dvipsnames,svgnames]{xcolor}		

\colorlet{MidnightBlue}{black!50!blue}
\colorlet{RoyalBlue}{black!20!blue}
\colorlet{OliveGreen}{black!60!green}
\colorlet{MyRed}{FireBrick!50!Crimson}
\colorlet{MyBlue}{DodgerBlue!60!black}
\colorlet{MyGreen}{DarkGreen!85!black}
\colorlet{MyLightBlue}{DodgerBlue!15}
\colorlet{MyLightGreen}{MyGreen!20}
\colorlet{LinkColor}{MediumBlue}
\colorlet{PrimalColor}{MyBlue}
\colorlet{PrimalFill}{MyLightBlue}
\colorlet{DualColor}{MyRed}

\newcommand{\afterhead}{.\;}		
\newcommand{\para}[1]{\medskip\paragraph{\textbf{#1\afterhead}}}

\usepackage{cancel}		
\usepackage{latexsym}		

\usepackage{pifont}		

\usepackage{tikz}		
\usetikzlibrary{calc,patterns,decorations.markings}		
\tikzset{negated/.style={
        decoration={markings,
            mark= at position 0.5 with {
                \node[transform shape,yshift=1pt] (tempnode) {\scriptsize$\xcancel{\quad}$};
            }
        },
        postaction={decorate}
    }
}

\usepackage{array}		
\usepackage{booktabs}		
\usepackage[inline,shortlabels]{enumitem}		
\setenumerate{itemsep=\smallskipamount,topsep=\smallskipamount,left=\parindent}
\setitemize{itemsep=\smallskipamount,topsep=\smallskipamount,left=\parindent}

\usepackage[kerning=true]{microtype}		

\usepackage{tabto}		
\usepackage{twoopt}		
\usepackage{xspace}		

\usepackage[sort&compress]{natbib}		

\bibpunct[, ]{[}{]}{,}{n}{,}{,}

\newcommand{\citefull}[2][]{\citeauthor{#2}, \citeyear[#1]{#2}}

\usepackage{hyperref}
\hypersetup{
colorlinks=true,
linktocpage=true,
pdfstartview=FitH,
breaklinks=true,
pdfpagemode=UseNone,
pageanchor=true,
pdfpagemode=UseOutlines,
plainpages=false,
bookmarksnumbered,
bookmarksopen=false,
bookmarksopenlevel=1,
hypertexnames=true,
pdfhighlight=/O,
urlcolor=LinkColor,linkcolor=LinkColor,citecolor=LinkColor,	
pdftitle={},
pdfauthor={},
pdfsubject={},
pdfkeywords={},
pdfcreator={pdfLaTeX},
pdfproducer={LaTeX with hyperref}
}

\newcommand{\EMAIL}[1]{\email{\href{mailto:#1}{#1}}}

\usepackage[sort&compress,capitalize,nameinlink]{cleveref}		
\crefname{algo}{Algorithm}{Algorithms}
\crefname{assumption}{Assumption}{Assumptions}
\crefname{hypothesis}{Hypothesis}{Hypotheses}
\crefname{method}{Method}{Methods}

\usepackage{algorithm}		
\usepackage{algpseudocode}		

\usepackage{thmtools}		
\usepackage{thm-restate}		

\theoremstyle{plain}
\newtheorem{theorem}{Theorem}		
\newtheorem{corollary}{Corollary}		
\newtheorem{lemma}{Lemma}		
\newtheorem{proposition}{Proposition}		

\newtheorem*{corollary*}{Corollary}		

\theoremstyle{definition}
\newtheorem{definition}{Definition}		
\newtheorem{assumption}{Assumption}		
\newtheorem{example}{Example}		
\newtheorem{algo}{Algorithm}		

\newtheorem*{definition*}{Definition}		
\newtheorem*{assumption*}{Assumption}		
\newtheorem*{blanket*}{Blanket assumption}		
\newtheorem*{example*}{Example}		

\theoremstyle{remark}
\newtheorem{remark}{Remark}		

\newtheorem*{remark*}{Remark}		

\def\endenv{\hfill{\small$\lozenge$}\medskip}

\renewcommand{\qedsymbol}{$\blacksquare$}		

\newcounter{proofpart}

\numberwithin{example}{section}		

\usepackage[showdeletions,suppress]{color-edits}		
\usepackage[normalem]{ulem}		

\newcommand{\debug}[1]{#1}		

\newcommand{\newmacro}[2]{\newcommand{#1}{\debug{#2}}}		
\newcommand{\newop}[2]{\DeclareMathOperator{#1}{\debug{#2}}}		

\DeclarePairedDelimiter{\braces}{\{}{\}}		
\DeclarePairedDelimiter{\bracks}{[}{]}		
\DeclarePairedDelimiter{\parens}{(}{)}		

\DeclarePairedDelimiter{\abs}{\lvert}{\rvert}		

\DeclarePairedDelimiterX{\setdef}[2]{\{}{\}}{#1:#2}		
\DeclarePairedDelimiterXPP{\exclude}[1]{\mathopen{}\setminus}{\{}{\}}{}{#1}

\newcommand{\N}{\mathbb{N}}		
\newcommand{\R}{\mathbb{R}}		

\DeclareMathOperator*{\intersect}{\bigcap}		
\DeclareMathOperator*{\union}{\bigcup}		

\DeclareMathOperator{\bigoh}{\mathcal{O}}		
\DeclareMathOperator{\cl}{cl}		
\DeclareMathOperator{\dist}{dist}		
\DeclareMathOperator{\ess}{ess}		
\DeclareMathOperator{\grad}{\nabla}		
\DeclareMathOperator{\Hess}{Hess}		
\DeclareMathOperator{\one}{\mathds{1}}		
\DeclareMathOperator{\relint}{ri}		

\newcommand{\cf}{cf.\xspace}		
\newcommand{\eg}{e.g.,\xspace}		
\newcommand{\ie}{i.e.,\xspace}		
\newcommand{\vs}{vs.\xspace}		

\newcommand{\textpar}[1]{\textup(#1\textup)}		

\newcommand{\alt}[1]{#1'}		
\newcommand{\altalt}[1]{#1''}		

\newmacro{\dd}{\:d}		
\newcommand{\ddt}{\frac{d}{dt}}		
\newcommand{\del}{\partial}		
\newcommand{\eps}{\varepsilon}		

\newcommand{\insum}{\sum\nolimits}		

\newmacro{\const}{c}		
\newmacro{\Const}{C}		
\newmacro{\coef}{\lambda}		

\newmacro{\param}{\theta}		
\newmacro{\params}{\Theta}		

\newmacro{\pexp}{p}		
\newmacro{\qexp}{q}		
\newmacro{\rexp}{r}		

\newmacro{\beforestart}{0}		
\newmacro{\start}{1}		
\newmacro{\afterstart}{2}		
\newmacro{\running}{\start,\afterstart,\dotsc}		
\newmacro{\halfrunning}{1/2,1,3/2,\dotsc}		

\newmacro{\run}{n}		
\newmacro{\runalt}{k}		
\newmacro{\nRuns}{T}		
\newmacro{\runs}{\mathcal{\nRuns}}		

\newmacro{\state}{X}		
\newmacro{\statealt}{Y}		
\newmacro{\statealtalt}{Z}		

\newcommand{\init}[1][\state]{\debug{#1}_{\start}}		

\newcommand{\iter}[1][\state]{\debug{#1}_{\runalt}}		

\newcommand{\prev}[1][\state]{\debug{#1}_{\run-1}}		
\newcommand{\curr}[1][\state]{\debug{#1}_{\run}}		
\renewcommand{\next}[1][\state]{\debug{#1}_{\run+1}}		

\newcommand{\beforelead}[1][\state]{\debug{#1}_{\run-1/2}}		
\newcommand{\lead}[1][\state]{\debug{#1}_{\run+1/2}}		

\newmacro{\tstart}{0}		
\newmacro{\timealt}{s}		
\newmacro{\horizon}{T}		

\newmacro{\imflow}{\Upsilon}		

\newmacro{\basin}{\mathcal{B}}		
\newmacro{\dbasin}{\mathcal{W}}		
\newmacro{\limset}{\mathcal{L}}		

\DeclarePairedDelimiterXPP{\flowof}[2]{\flow_{#1}}{(}{)}{}{#2}		
\DeclarePairedDelimiterXPP{\imflowof}[2]{\imflow_{#1}}{(}{)}{}{#2}		

\newop{\Nash}{NE}		
\newop{\CE}{CE}		
\newop{\CCE}{CCE}		
\newop{\NI}{NI}		

\newop{\brep}{br}		
\newop{\reg}{Reg}		
\newop{\preg}{\overline{Reg}}		
\newop{\val}{val}		

\newmacro{\play}{i}		
\newmacro{\playalt}{j}		
\newmacro{\playaltalt}{k}		
\newmacro{\nPlayers}{N}		
\newmacro{\players}{\mathcal{\nPlayers}}		

\newmacro{\pure}{\alpha}		
\newmacro{\purealt}{\beta}		
\newmacro{\purealtalt}{\gamma}		
\newmacro{\nPures}{A}		
\newmacro{\pures}{\mathcal{\nPures}}		

\newmacro{\strat}{x}		
\newmacro{\stratalt}{\alt\strat}		
\newmacro{\strataltalt}{\altalt\strat}		
\newmacro{\strats}{\mathcal{X}}		
\newmacro{\intstrats}{\strats^{\circle}}		

\newcommand{\eq}{\sol}		

\newmacro{\loss}{\ell}		
\newmacro{\pay}{u}		
\newmacro{\payv}{v}		
\newmacro{\pot}{\Phi}		

\newmacro{\game}{\mathcal{G}}		
\newmacro{\gamefull}{\game(\players,\points,\pay)}		

\newmacro{\fingame}{\Gamma}		
\newmacro{\fingamefull}{\Gamma(\players,\pures,\pay)}		
\newmacro{\mixgame}{\simplex(\fingame)}

\newmacro{\gmat}{g}		
\newmacro{\gdist}{\dist_{\gmat}}
\newmacro{\mfld}{\mathcal{M}}		
\newmacro{\form}{\omega}		

\newmacro{\tvec}{z}		
\newmacro{\uvec}{u}		

\newmacro{\ball}{\mathbb{B}}		
\newmacro{\sphere}{\mathbb{S}}		

\newmacro{\vertex}{v}		
\newmacro{\vertexalt}{\alt\vertex}		
\newmacro{\vertexaltalt}{\altalt\vertex}		
\newmacro{\nVertices}{V}		
\newmacro{\vertices}{\mathcal{\nVertices}}		

\newmacro{\edge}{e}		
\newmacro{\edgealt}{\alt\edge}		
\newmacro{\edgealtalt}{\altalt\edge}		
\newmacro{\nEdges}{E}		
\newmacro{\edges}{\mathcal{\nEdges}}		

\newmacro{\graph}{\mathcal{G}}		
\newmacro{\graphall}{\graph(\vertices,\edges)}		

\newmacro{\vecspace}{\mathcal{Z}}		
\newmacro{\subspace}{\mathcal{W}}		

\newmacro{\bvec}{e}		
\newmacro{\bvecs}{\mathcal{E}}		

\newmacro{\coord}{k}		
\newmacro{\coordalt}{j}		
\newmacro{\coordaltalt}{l}		
\newmacro{\nCoords}{d}		
\newmacro{\dims}{\nCoords}		
\newmacro{\vdim}{\nCoords}		

\newmacro{\pvec}{z}		
\newmacro{\dvec}{w}		

\newmacro{\pspace}{\mathcal{V}}		
\newmacro{\dspace}{\mathcal{Y}}		

\newmacro{\ppoint}{x}		
\newmacro{\ppointalt}{\alt\ppoint}		
\newmacro{\ppointaltalt}{\altalt\ppoint}		
\newmacro{\ppoints}{\mathcal{X}}		
\newmacro{\pstate}{X}		

\newmacro{\dpoint}{y}		
\newmacro{\dpointalt}{\alt\dpoint}		
\newmacro{\dpointaltalt}{\altalt\dpoint}		
\newmacro{\dpoints}{\mathcal{Y}}		
\newmacro{\dstate}{\state}		

\newmacro{\dset}{\mathcal{D}}		
\newmacro{\dnhd}{\mathcal{D}}		
\newmacro{\dlyap}{F}		

\newmacro{\mat}{M}		
\newmacro{\hmat}{H}		

\newmacro{\ones}{\mathbf{1}}		
\newmacro{\eye}{I}		
\newmacro{\zer}{\mathbf{0}}		

\DeclarePairedDelimiter{\norm}{\lVert}{\rVert}		
\DeclarePairedDelimiterXPP{\dnorm}[1]{}{\lVert}{\rVert}{_{\ast}}{#1}		

\DeclarePairedDelimiterXPP{\onenorm}[1]{}{\lVert}{\rVert}{_{1}}{#1}		
\DeclarePairedDelimiterXPP{\twonorm}[1]{}{\lVert}{\rVert}{}{#1}		
\DeclarePairedDelimiterXPP{\supnorm}[1]{}{\lVert}{\rVert}{_{\infty}}{#1}		

\DeclarePairedDelimiterX{\braket}[2]{\langle}{\rangle}{#1,#2}		

\DeclarePairedDelimiterX{\inner}[2]{\langle}{\rangle}{#1,#2}		

\newcommand{\defeq}{\coloneqq}		
\newcommand{\eqdef}{\eqqcolon}		

\newcommand{\from}{\colon}		

\newop{\Opt}{Opt}		
\newop{\Sol}{Sol}		
\newop{\gap}{Gap}		

\newmacro{\orcl}{V}		
\newop{\err}{Err}		
\newop{\IWE}{IWE}		

\newmacro{\obj}{f}		
\newmacro{\objalt}{g}		
\newmacro{\sobj}{F}		

\newmacro{\gvec}{g}		
\newmacro{\gbound}{G}		

\newcommand{\sol}[1][\point]{#1^{\ast}}		

\newmacro{\vecfield}{v}		
\newmacro{\oper}{A}		
\newmacro{\vbound}{\gbound}		

\newmacro{\lips}{L}		
\newmacro{\strong}{\alpha}		
\newmacro{\smooth}{\beta}		

\newop{\tcone}{TC}		
\newop{\dcone}{\tcone^{\ast}}		
\newop{\ncone}{NC}		
\newop{\pcone}{PC}		
\newop{\hull}{\Delta}		

\newmacro{\cvx}{\mathcal{C}}		
\newmacro{\subd}{\partial}		

\newmacro{\minmax}{\ell}		

\newmacro{\minvar}{y}		
\newmacro{\minvaralt}{\alt\minvar}		
\newmacro{\minvars}{\mathcal{Y}}		
\newmacro{\minstate}{Y}		

\newmacro{\maxvar}{z}		
\newmacro{\maxvaralt}{\alt\maxvar}		
\newmacro{\maxvars}{\mathcal{Z}}		
\newmacro{\maxstate}{Z}		

\newop{\Eucl}{\Pi}		
\newop{\logit}{\Lambda}		
\newop{\dkl}{KL}		

\newmacro{\hreg}{h}		
\newmacro{\hconj}{\hreg^{\ast}}		
\newmacro{\breg}{D}		
\newmacro{\mprox}{P}		
\newmacro{\mirror}{Q}		
\newmacro{\fench}{F}		
\newmacro{\hstr}{K}		

\newmacro{\proxdom}{\points_{\hreg}}		
\newmacro{\proxdomi}{\points_{\hreg_{\play}}}		

\DeclarePairedDelimiterXPP{\proxof}[2]{\mprox_{#1}}{(}{)}{}{#2}		

\newmacro{\zone}{\mathbb{D}}		

\newmacro{\point}{x}		
\newmacro{\pointalt}{\alt\point}		
\newmacro{\pointaltalt}{\altalt\point}		
\newmacro{\points}{\mfld}		
\newmacro{\intpoints}{\relint\points}		

\newmacro{\base}{p}		
\newmacro{\basealt}{q}		
\newmacro{\basealtalt}{u}		

\newmacro{\open}{\mathcal{U}}		
\newmacro{\closed}{\mathcal{C}}		
\newmacro{\cpt}{\mathcal{K}}		

\newmacro{\set}{\mathcal{S}}		
\newmacro{\nhd}{\mathcal{U}}		

\newop{\ex}{\mathbb{E}}		
\newop{\prob}{\mathbb{P}}		
\newop{\Var}{Var}		
\newop{\simplex}{\hull}		

\providecommand\given{}		

\DeclarePairedDelimiterXPP{\exof}[1]{\ex}{[}{]}{}{
\renewcommand\given{\nonscript\,\delimsize\vert\nonscript\,\mathopen{}} #1}

\DeclarePairedDelimiterXPP{\probof}[1]{\prob}{(}{)}{}{
\renewcommand\given{\nonscript\:\delimsize\vert\nonscript\:\mathopen{}} #1}

\DeclarePairedDelimiterXPP{\oneof}[1]{\one}{(}{)}{}{
\renewcommand\given{\nonscript\,\delimsize\vert\nonscript\,\mathopen{}} #1}

\newmacro{\seed}{\sample}		
\newmacro{\seeds}{\samples}		
\newmacro{\pdist}{P}		
\newmacro{\history}{\mathcal{H}}		

\newmacro{\sample}{\omega}		
\newmacro{\samples}{\Omega}		

\newmacro{\filter}{\mathcal{F}}		
\newmacro{\probspace}{(\samples,\filter,\prob)}		

\newcommand{\as}{\debug{\textpar{a.s.}}\xspace}		
\newmacro{\event}{\mathcal{E}}       
\newmacro{\eventalt}{\mathcal{H}}       
\newmacro{\mean}{\mu}		
\newmacro{\sdev}{\sigma}		
\newmacro{\variance}{\sdev^{2}}		

\newcommand{\est}[1]{\hat #1}		

\newmacro{\signal}{V}		
\newmacro{\step}{\gamma}		
\newmacro{\learn}{\eta}		

\newmacro{\efftime}{\tau}		
\newcommand{\apt}[2][]{\dstate_{#1}(#2)}		

\newmacro{\noise}{U}		
\newmacro{\snoise}{\xi}		
\newmacro{\noisepar}{\sdev}		
\newmacro{\noisevar}{\variance}		
\newmacro{\aggnoise}{\mathrm{\uppercase\expandafter{\romannumeral1}}}		
\newmacro{\supnoise}{\aggnoise_{\infty}}		
\newmacro{\maxnoise}{\aggnoise^{\ast}}		

\newmacro{\bias}{b}		
\newmacro{\bbound}{B}		
\newmacro{\sbias}{\chi}		
\newmacro{\aggbias}{\mathrm{\uppercase\expandafter{\romannumeral2}}}		
\newmacro{\supbias}{\aggbias_{\infty}}		
\newmacro{\maxbias}{\aggbias^{\ast}}		

\newmacro{\second}{\psi}		
\newmacro{\sbound}{M}		
\newmacro{\aggsecond}{\mathrm{\uppercase\expandafter{\romannumeral3}}}		
\newmacro{\supsecond}{\aggsecond_{\infty}}		
\newmacro{\maxsecond}{\aggsecond^{\ast}}		

\newmacro{\mix}{\delta}		
\newmacro{\unitvec}{w}		
\newmacro{\unitvar}{W}		
\newmacro{\perturb}{z}		

\newmacro{\purequery}{\est\pure}		
\newmacro{\query}{\est\state}		
\newmacro{\pivot}{\point}		
\newmacro{\querypoint}{\est\point}		

\addauthor[Pan]{PM}{MediumBlue}

\DeclarePairedDelimiterXPP{\expof}[2]{\exp_{#1}}{(}{)}{}{#2}

\newmacro{\pureq}{\eq[\pure]}

\newop{\probalt}{\mathbb{Q}}		
\newmacro{\good}{\event}
\newmacro{\bad}{\mathcal{N}}
\newmacro{\lyap}{\pot}
\newmacro{\gauge}{\varphi}

\newmacro{\energy}{E}
\newmacro{\elvl}{a}
\newmacro{\ebound}{H}
\newmacro{\esmooth}{\beta}
\newmacro{\emax}{\energy_{\ast}}
\newmacro{\hien}{\energy_{+}}
\newmacro{\loen}{\energy_{-}}
\newmacro{\ediff}{\Delta\elvl}

\newmacro{\lvl}{c}

\newmacro{\thres}{\eps}		
\newmacro{\conf}{\rho}		
\newmacro{\stoptime}{N}		

\newmacro{\texp}{\mu}		
\newmacro{\bexp}{b}		
\newmacro{\sexp}{s}		

\newmacro{\dev}{z}		
\newmacro{\iDev}{k}		
\newmacro{\devs}{\mathcal{Z}}		

\newmacro{\score}{\dpoint}		
\newmacro{\cone}{\mathcal{C}}		
\newmacro{\poly}{\mathcal{P}}		

\newmacro{\termi}{\mathrm{\uppercase\expandafter{\romannumeral1}}}		
\newmacro{\termii}{\mathrm{\uppercase\expandafter{\romannumeral2}}}		
\newmacro{\termiii}{\mathrm{\uppercase\expandafter{\romannumeral3}}}		

\newmacro{\conj}{\mathcal{C}}     
\newmacro{\nconj}{\mathcal{H}} 
\newmacro{\pcoord}{\shortparallel}     

\newcommand{\normal}[1]{\tilde{#1}}

\newmacro{\fdiffeo}{\normal{\Phi}}     
\newmacro{\fnbhd}{\normal{\nhd}}     
\newmacro{\fpoint}{\normal{\point}}     
\newmacro{\fcurve}{\normal{\curve}}     
\newmacro{\fvecfield}{\normal{\vecfield}}     
\newmacro{\pvecfield}{\normal{\vecfield}^{\pcoord}}     

\newcommand{\fflow}[1][h]{\normal{\flowcurve}(#1)}     
\newcommand{\fflowc}[1][h]{\normal{\flowcurve}^\coord(#1)}     
\newcommand{\fpflow}[1][h]{\normal{\Pflowmap}(#1)}     
\newcommand{\faptbar}[1][u]{\normal{\bar\point}(#1)}

\newcommand{\dfflow}[1][h]{\dot{\normal{\flowcurve}}(#1)}     
\newcommand{\dfflowc}[1][u]{\dot{\normal{\flowcurve}}^{\coord}(#1)}     

\newcommand{\dfpflow}[1][h]{\dot{\normal{\Pflowmap}}(#1)}     
\newcommand{\dfpflowc}[1][u]{\dot{\normal{\Pflowmap}}^\coord(#1)}     

\newmacro{\cptrans}{\normal{\Lambda} \left(u,\fpflow[u] \right)}	
\newmacro{\cptransc}{\normal{\Lambda}_{\coord}\left(u,\fpflow[u] \right)}   

\newmacro{\cptransp}{\normal{\Lambda}^{\pcoord}\left(u,\fpflow[u] \right)}	
\newmacro{\cptranscp}{\normal{\Lambda}^{\pcoord}_{\coord}\left(u,\fpflow[u] \right)}

\newmacro{\minind}{i}     
\newmacro{\maxind}{j}     

\newmacro{\degr}{k}       
\newmacro{\formalt}{\lambda}     
\newmacro{\gform}{\zeta}     
\newmacro{\volform}{\textup{vol}}     

\newmacro{\Hstar}{\star}		
\newmacro{\diff}{\mathrm{d}}		
\newmacro{\codiff}{\delta}		

\newmacro{\dirac}{\iota}     
\newmacro{\convo}{\ast}	     

\newmacro{\GreenOp}{\mathbb{G}}		
\newmacro{\fundsol}{\Phi}		     
\newmacro{\fundconst}{\frac{1}{\vdim(\vdim-2)  \abs{\ball_{\vdim}}}}     
\newmacro{\fundconstd}{\frac{1}{\vdim  \abs{\ball_{\vdim}}}}     
\newmacro{\fundconstdd}{\frac{-1}{\abs{\ball_{\vdim}}}}     
\newmacro{\avgdmax}{\bar{U}}     

\newmacro{\Lap}{\bigtriangleup}		
\newmacro{\HdRLap}{\bigtriangleup_{\texttt{\tiny HdR}}}		
\newmacro{\minLap}{\bigtriangleup_{\minvar}}     
\newmacro{\maxLap}{\bigtriangleup_{\maxvar}}     
\newmacro{\minmaxLap}{\myhg}	
\newmacro{\Har}{\mathbb{H}}     
\newmacro{\Ker}{\mathbb{K}}     

\newmacro{\bgame}{{ \diff\minvar^{\minind} \wedge \diff\maxvar^{\maxind} }}     

\newmacro{\stochfield}{\mathsf{V}}		

\newmacro{\bindex}{I}	
\newmacro{\compbindex}{{\bar{I}}}	

\newmacro{\lyapalt}{ \hat\lyap}     
\newmacro{\colyap}{\tilde\lyap}     

\newmacro{\error}{W}		

\newmacro{\ctime}{t}
\newmacro{\ctimealt}{s}
\newmacro{\cstart}{0}		

\newcommand{\radial}{r}		
\newmacro{\sset}{\mathcal{S}}		

\newmacro{\flowmap}{\Theta}		
\newmacro{\flowcurve}{\theta}		
\newcommandtwoopt{\flow}[2][\ctime][\point]{\flowmap_{#1}(#2)}

\newcommand{\orbit}[2][]{\flowcurve_{#1}(#2)}		
\newcommand{\dotorbit}[2][]{\dot\flowcurve_{#1}(#2)}		

\newmacro{\Pflowmap}{\lambda}		
\newcommandtwoopt{\Pflow}[2][\ctime][\point]{\Pflowmap_{#1}(#2)}

\newmacro{\Ent}{\textup{Ent}}

\def\drm{{ \mathrm{d} }}

\DeclareMathOperator{\ctgh}{ctgh}       

\newmacro{\vvec}{v}		
\newmacro{\vvecalt}{w}		

\newmacro{\mfd}{\mathcal{M}}		
\newmacro{\curve}{\gamma}          
\newcommand{\ptrans}[3]{\Gamma_{#1\to #2} \left( #3 \right)}      
\newmacro{\sect}{K}    

\newmacro{\rgrad}{ \nabla }

\newmacro{\retrbase}{{\mathcal{R}}}

\newmacro{\nbhd}{\mathcal{U}}		
\newmacro{\tinv}{m}		
\newmacro{\nn}{\nonumber}		
\newmacro{\rbound}{R}		

\newmacro{\noiselevel}{\sqrt{\run \log^{1+\frac{\epsilon}{2}} \run}}

\newmacro{\aux}{\tilde\lyap}		
\newcommand{\tspace}[1][\point]{\mathcal{T}_{#1} \points  }		
\newmacro{\tbundle}{ \mathcal{T}\points  }		

\addauthor[Ya-Ping]{YPH}{DarkGreen}

\newcommand{\coer}{R}
\newcommand{\ssstyle}{}

\newacro{LHS}{left-hand side}
\newacro{RHS}{right-hand side}
\newacro{iid}[i.i.d.]{independent and identically distributed}
\newacro{lsc}[l.s.c.]{lower semi-continuous}

\newacro{APT}{asymptotic pseudotrajectory}
\newacroplural{APT}[APTs]{asymptotic pseudotrajectories}
\newacro{GD}{gradient dynamics}
\newacro{GF}{gradient flow}
\newacro{ICT}{internally chain-transitive}
\newacro{MDS}{martingale difference sequence}
\newacro{NE}{Nash equilibrium}
\newacroplural{NE}[NE]{Nash equilibria}
\newacro{NSE}{Nash\textendash Stampacchia equilibrium}
\newacroplural{NSE}[NSE]{Nash\textendash Stampacchia equilibria}
\newacro{ODE}{ordinary differential equation}
\newacro{RMD}{Riemannian mean dynamics}
\newacro{SA}{stochastic approximation}
\newacro{SFO}{stochastic first-order oracle}
\newacro{SG}{stochastic gradient}
\newacro{SP}{saddle-point}
\newacro{WAC}{weak asymptotic coercivity}

\newacro{AH}{Arrow\textendash Hurwicz}
\newacro{BDG}{Burkholder\textendash Davis\textendash Gundy}
\newacro{ConO}{consensus optimization}
\newacro{RM}{Robbins\textendash Monro}
\newacro{KW}{Kiefer\textendash Wolfowitz}
\newacro{GDA}{gradient descent/ascent}
\newacro{SGA}{symplectic gradient adjustment}
\newacro{SGD}{stochastic gradient descent}
\newacro{SGDA}{stochastic gradient descent/ascent}
\newacro{SPSA}{simultaneous perturbation stochastic approximation}
\newacro{ASGDA}[alt-SGDA]{alternating stochastic gradient descent/ascent}
\newacro{SEG}{stochastic extra-gradient}
\newacro{EG}{extra-gradient}
\newacro{PEG}{Popov's extra-gradient}
\newacro{RG}{reflected gradient}
\newacro{OG}{optimistic gradient}
\newacro{PPM}{proximal point method}

\newacro{HdR}{Hodge-de Rham Laplacian}
\newacro{HD}{Hodge decomposition}

\newacro{GAN}{generative adversarial network}
\newacro{NN}{neural network}
\newacro{FTRL}{``follow the regularized leader''}
\newacro{CGD}{Competitive Gradient Descent}
\newacro{wp1}[w.p.$1$]{with probability $1$}

\newacro{RRM}[RRM]{Riemannian Robbins\textendash Monro}

\newacro{RGDA}[RGDA]{Riemannian gradient descent(/ascent)}
\newacro{RSGD}[RSGD]{Riemannian stochastic gradient descent}
\newacro{REG}[REG]{Riemannian extra-gradient}
\newacro{RSEG}[RSEG]{Riemannian stochastic extra-gradient}
\newacro{RPEG}[RPEG]{Riemannian Popov's extra-gradient}
\newacro{RRG}[RRG]{Riemannian reflected gradient}
\newacro{ROG}[ROG]{Riemannian optimistic gradient}
\newacro{RPPM}[RPPM]{Riemannian proximal point method}
\newacro{RSGDA}[RSGDA]{Riemannian stochastic gradient descent(/ascent)}
\newacro{RASGDA}[R-alt-SGDA]{Riemannian alternating stochastic gradient descent/ascent}
\newacro{RKW}[RKW]{Riemannian Kiefer\textendash Wolfowitz}
\newacro{RAPT}[RAPT]{Riemannian asymptotic pseudotrajectory}

\newacro{NPG}{natural policy gradient}
\newacro{PFS}{parallel frame system}
\newacro{FCS}{Fermi coordinate system}
\newacro{NGD}{Natural Gradient Descent}
\newacro{RSGM}[RSGM]{Riemannian stochastic gradient method}
\newacro{RSCEG}{Riemannian stochastic corrected extragradient}

\begin{document}


\title{Riemannian stochastic approximation algorithms}		

\author
[M.~R.~Karimi]
{Mohammad Reza Karimi$^{\ast}$}
\address{$^{\ast}$\,%
Institute for Machine Learning, CAB G 65, Universitaetstrasse 6, 8092 Zurich, Switzerland.}
\EMAIL{\{mkarimi,yaping.hsieh\}@inf.ethz.ch}
\author
[Y.~P.~Hsieh]
{Ya-Ping Hsieh$^{\ast}$}
\author
[P.~Mertikopoulos]
{\\Panayotis Mertikopoulos$^{\diamond}$}
\address{$^{\diamond}$\,%
Univ. Grenoble Alpes, CNRS, Inria, Grenoble INP, LIG, 38000 Grenoble, France.}
\EMAIL{panayotis.mertikopoulos@imag.fr}
\author
[A.~Krause]
{Andreas Krause$^{\ast}$}
\EMAIL{krausea@inf.ethz.ch}



\thanks{A one-page extended abstract of this paper was presented at COLT 2022.}

\begin{abstract}
%
%
We examine a wide class of stochastic approximation algorithms for solving (stochastic) nonlinear problems on Riemannian manifolds.
Such algorithms arise naturally in the study of Riemannian optimization, game theory and optimal transport, but their behavior is much less understood compared to the Euclidean case because of the lack of a global linear structure on the manifold.
We overcome this difficulty by introducing a suitable Fermi coordinate frame which allows us to map the asymptotic behavior of the \acf{RRM} algorithms under study to that of an associated deterministic dynamical system.
In so doing, we provide a general template of almost sure convergence results that mirrors and extends the existing theory for Euclidean \acl{RM} schemes, despite the significant complications that arise due to the curvature and topology of the underlying manifold.
We showcase the flexibility of the proposed framework by applying it to a range of retraction-based variants of the popular optimistic / extra-gradient methods for solving minimization problems and games, and we provide a unified treatment for their convergence.
\end{abstract}
\maketitle

\allowdisplaybreaks		
\acresetall		

\section{Introduction}
\label{sec:introduction}

Consider a nonlinear system of equations of the general form
\begin{equation}
\label{eq:root}
\tag{Root}
\text{Find $\sol\in\points$ such that $\vecfield(\sol) = 0$}
\end{equation}
where $\points$ is a smooth $\vdim$-dimensonal manifold and $\vecfield$ is a vector field on $\points$.
Root-finding problems of this type play a crucial role in many areas of mathematical programming and data science, from Riemannian optimization and game theory to reinforcement learning, signal processing,
and information theory.
In particular, in addition to standard minimization problems \textendash\ that is, when $\vecfield = -\nabla\obj$ for some smooth function $\obj$ on $\points$ \textendash\ the general form of \eqref{eq:root} includes bilevel and saddle-point problems, dynamic programming, games in normal form, and many other equilibrium problems that arise in practice.
For a range of applications and a comprehensive introduction to the topic, see \citet{AMS08}, \citet{Bou22}, and references therein.

The vast majority of methods for solving \eqref{eq:root} are iterative in nature, and they rely on building a successively finer ``model'' function which is applied to the last computed approximation of a root in order to get a new approximation.
Usually, this model function is based on the value of $\vecfield$ at a candidate solution;%
\footnote{More sophisticated methods \textendash\ like Newton's algorithm \textendash\ rely on a first-order Taylor approximation of $\vecfield$ around an iterate;
Halley's method employs a second-order model, and the hierarchy continues with the general class of Householder methods that use $k$-th order derivatives to build more precise polynomial models of $\vecfield$ around an iterate.
However, because these tensor methods involve the computation of higher-order derivatives of $\vecfield$, they may become highly impractical for moderate-to-high values of $\vdim$.}
in many cases however, even this first-order model is too costly or even impossible to compute \textendash\ \eg if $\vecfield(\point) = \exof{\orcl(\point;\sample)}$ for some random variable $\sample$ with unknown distribution.
In this case, a popular alternative is to rely on \acl{SA} algorithms that only require oracle access to a random \textendash\ and possibly incomplete \textendash\ approximation of $\vecfield$ at the queried point.

In this general context, when $\points = \R^{\vdim}$, the method of choice for solving \eqref{eq:root} is the \acdef{RM} algorithm
\begin{equation}
\label{eq:RM}
\tag{RM}
\next
    = \curr
		+ \curr[\step] \curr[\signal]
\end{equation}
where $\curr[\step] > 0$ is a variable step-size sequence and $\curr[\signal]$ is a random estimate of $\vecfield(\curr)$.
This method was introduced in the seminal papers of \citet{RM51} and \citet{KW52}, and the first general convergence results were obtained by \citet{Lju77,Lju78} for gradient problems.
This has subsequently led to substantial activity on the topic, with major contributions by \citet{BH96,BMP90, KC78,KY97}, and many others.
However, the linear structure of $\R^{\vdim}$ is deeply ingrained in all these works \textendash\ and the method itself \textendash\ preventing its use for solving \eqref{eq:root} in a manifold setting \textendash\ \eg the $\vdim$-dimensional torus for a robotic arm with $\vdim$ joints,
Grassman or Stiefel manifolds for robust principal component analysis,
hyperbolic space for text and graph embeddings, etc. 
Because of this, the applicability of \ac{RM} methods to general root-finding problems is significantly more narrow than one might expect, even when $\points$ has a relatively simple structure (like a matrix group or a Grassmannian).

\para{Our contributions in the context of related work}

In view of the above, our main objective is to bridge the gap between Euclidean and Riemannian \acl{SA} schemes by replacing the ``$+$'' operation in \eqref{eq:RM} with the Riemannian exponential map $\exp_{\curr}(\cdot)$ on $\points$ \textendash\ or, more generally (and often more tractably), a \emph{retraction} on $\points$ based at $\curr$.
In Riemannian optimization, this approach was pioneered by \citet{Bon13}, who examined the case where $\vecfield$ is the Riemannian gradient of some objective function $\obj$.
Subsequent works \cite{zhang2016first,tripuraneni2018averaging,boumal2019global,
criscitiello2019efficiently,lezcano2020curvature,wang2021no} expanded on the results of \citet{Bon13} for \acl{RSGD}, while similar results were obtained in \cite{ferreira2002proximal, li2009monotone, bento2017iteration, huang2021riemannian} for \aclp{RPPM}.%

All these works focus exclusively on the case where $\vecfield$ is a gradient field, so they do not apply to general, non-gradient instances of \eqref{eq:root}.
A partial extension to the non-gradient case was provided by a concurrent line of works which examined the use of \acl{REG} methods under the assumption of (geodesic) monotonicity \cite{ferreira2005singularities, tang2012korpelevich, neto2016extragradient, fan2020tseng, khammahawong2020extragradient, chen2021modified}.
This is a strong, convexity-type assumption which posits that $\vecfield$ globally points towards its (necessarily connected) root system in a suitable, geodesic sense;
convergence is then obtained following a similar line of reasoning as in the case of monotone operator theory in Hilbert spaces \citep{BC17}.

Our paper does not make any such assumptions and directly examines the dynamics of \acl{RRM} methods for general vector fields $\vecfield$.
In this regard, our main contributions can be summarized as follows:
\begin{enumerate}
\item
We introduce a \emph{generalized} \acli{RRM} template which includes as special cases all methods mentioned above (Riemannian \acl{SGD}, \acl{EG}, \aclp{PPM}, etc.), as well as a number of new \ac{SA} schemes for \eqref{eq:root}.
\item
Under some mild technical conditions on $\points$, we show that the sequence of generated points forms an ``approximate solution'' \textendash\ an \ac{APT} to be exact \textendash\ of an associated deterministic dynamical system (\cref{thm:APT}).
\end{enumerate}
This \acl{SA} result extends the seminal theory of \citet{BH96} for Euclidean \acl{RM} schemes to a Riemannian setting, and allows us to infer almost sure convergence of \ac{RRM} schemes to the \ac{ICT} sets of the underlying Riemannian dynamics (\cf \cref{cor:ICT,cor:ICT-expl}).
In gradient and strictly monotone problems, these \ac{ICT} sets are precisely the roots of $\vecfield$, so we readily recover many of the asymptotic convergence results mentioned above (often under much weaker assumptions).
In addition, as we show in \cref{sec:applications}, our framework applies to several settings beyond gradient or monotone systems \textendash\ such as ordinal potential games, supermodular games, and cooperative dynamics \textendash\ and covers a significantly wider class of \ac{SA} schemes.

\para{Tools, techniques, and related approaches}

In the absence of a linear structure on $\points$, the major challenge we have to overcome is the lack of a suitable coordinate frame within which to analyze the trajectories of Riemannian \ac{SA} algorithms.
This reflects the dichotomy that, unlike in the case of $\R^\vdim$, points and vectors on manifolds obey fundamentally different rules and have to be compared using different moving frames.
To circumvent this obstacle, we introduce an \emph{extended Fermi coordinate} frame inspired by \citet{manasse1963fermi}, and we use it to prove that Riemannian \ac{SA} schemes enjoy similar error bounds as in Euclidean spaces (up to some high-order terms that vanish in the long run).
The aggregation and propagation of these errors can then be controlled using arguments from martingale limit theory which ultimately yield the convergence properties mentioned above.

A concurrent approach to establish the \ac{APT} property in Riemannian \ac{SA} schemes is due to \citet{shah2021stochastic}, who assumes the existence of a local diffeomorphism mapping geodesic interpolations to linear interpolations in a Euclidean space. 
However, the existence of such a diffeomorphism on every point of $\points$ implies that the manifold is globally \emph{flat}, \ie essentially Euclidean \citep{iliev2006handbook};
this assumption is too restrictive for bona fide Riemannian applications, so the analysis of \cite{shah2021stochastic} is not relevant for our purposes.
An additional issue is that the error bounds employed by \citet[p. 1131]{shah2021stochastic} rule out vector fields with a rotational component \textendash\ such as $\vecfield(x,y) = (-y,x)$ on $\R^2$ \textendash\,  further limiting the applicability of their techniques to the setting under consideration.

Finally, \citet{durmus2020convergence,durmus2021riemannian} also recently consider a generic version of \acl{RM} schemes, with both vanishing and constant step-sizes.
The analysis of the latter type of schemes cannot lead to convergence with probability $1$, so the results of \cite{durmus2021riemannian} are necessarily ergodic in nature and hence beyond our paper's scope.
The setting of \cite{durmus2020convergence} is closer in spirit to our own, and it also accounts for the effects of bias in the queries to $\vecfield$;
however, the results obtained therein concern dynamics that admit a \emph{Lyapunov function} \textendash\ the so-called ``gradient-like'' case \citep{Ben99} \textendash\ so there is no overlap with our analysis.

\para{Basic notions and notation}

Throughout our paper, $\points$ will denote a $\vdim$-dimensional, geodesically complete, Riemannian manifold.
We will write $\inner{\cdot}{\cdot}_\point$, $\point\in\points$, for its underlying metric, $\norm{\cdot}_{\point}$ for the induced norm, and we will assume that the sectional curvatures of $\points$ are bounded from above and below by $K_\textup{up}$ and $K_\textup{low}$ respectively.

For any curve $\curve$ on $\points$, the notation $\dot\curve(\ctime)$ will always denote its velocity at time $\ctime$.
Given any pair of points $\point, \pointalt \in \points$ and a vector $\vvec$ in the tangent space
at $\point$, $\tspace$, we denote by $\ptrans{\point}{\point'}{\vvec} \in \tspace[\point']$ the vector obtained by parallel transporting $\vvec$ along the minimizing geodesic connecting $\point$ and $\point'$;
if minimizing geodesics are not unique, $\ptrans{\point}{\point'} \vvec $ will be understood as parallel transport along any of them.
We also say that a vector field $\vecfield$ on $\points$ is (geodesically) $\lips$-Lipschitz if, for all $\point, \point' \in \points$, we have
\begin{equation}
\norm{ \ptrans{\point}{\point'}{\vecfield(\point)} - \vecfield(\point')  }_{\point'}
	= \norm{ \vecfield(\point)  - \ptrans{\point'}{\point}{\vecfield(\point')}}_{\point }
	\leq \lips \dist(\point, \point'),
\end{equation}
where $\dist(\cdot,\cdot)$ denotes the distance function on $\points$ induced by $\inner{\cdot}{\cdot}$.
All vector fields in this paper are assumed to be $\lips$-Lipschitz, complete and bounded, \ie $\vbound \defeq \sup_{\point\in\points} \norm{\vecfield(\point)}_\point < \infty$.
Finally, the Riemannian gradient of a smoooth function $\obj\from\points\to\R$ will be denoted by $\nabla\obj$.

For a detailed account, we refer the reader to the masterful treatment of \citet{Lee97}.
\acresetall
\acused{wp1}

\section{Riemannian Robbins\textendash Monro algorithms: Definitions and assumptions}
\label{sec:RRM}

\subsection{The \acl{RRM} template}

We begin by discussing the basic template of \acdef{RRM} algorithms.
As we stated before, the main difference with their Euclidean counterpart is that addition along ``straight lines'' is replaced with the Riemannian exponential mapping.
This leads to the abstract update rule
\begin{equation}
\label{eq:RRM}
\tag{RRM}
\next
	= \expof{\curr}{\curr[\step]\parens*{\vecfield(\curr) + \curr[\error]}}
\end{equation}
where
\begin{enumerate}
\item
$\curr \in \points$ denotes the state of the algorithm at each iteration counter $\run = \running$
\item
$\curr[\error] \in \tspace[{\scriptscriptstyle\curr}]$ is an abstract error term (described in detail below).
\item
$\curr[\step] > 0$ is the method's step-size (also discussed below).
\end{enumerate}
In the above, we tacitly assume that the error term $\curr[\error]$ is generated \emph{after}
$\curr$, so it is not adapted to the history $\curr[\filter] \defeq \sigma(\init,\dotsc,\curr)$ of $\curr$ \textendash\ that is, $\curr[\error]$ is random relative to $\curr[\filter]$.
In addition, to differentiate between ``random'' (zero-mean) and ``systematic'' (non-zero-mean) errors, it will be convenient to further decompose $\curr[\error]$ as
\begin{equation}
\label{eq:error}
\curr[\error]
	= \curr[\noise]
		+ \curr[\bias]
\end{equation}
where
$\curr[\noise] = \curr[\error] - \exof{\curr[\error] \given \curr[\filter]}$ captures the random, zero-mean part of $\curr[\error]$,
while
$\curr[\bias] = \exof{\curr[\error] \given \curr[\filter]}$ represents the systematic component thereof.

To quantify all this, we will assume in the sequel that $\curr[\noise]$ and $\curr[\bias]$ are bounded as
\begin{equation}
\label{eq:signal-stats}
\exof{\norm{\curr[\noise]}_{\scriptscriptstyle\curr}^{2} \vert \curr[\filter]}
	\leq \curr[\sdev]^{2}
	\quad
	\textrm{and}
	\quad
\exof{\norm{\curr[\bias]}_{\scriptscriptstyle\curr} \given \curr[\filter]}
	\leq \curr[\bbound]
  \quad \text{\as}
\end{equation}
where $\curr[\sdev]$ and $\curr[\bbound]$ are to be construed as upper bounds on the noise and bias of the error terms entering \eqref{eq:RRM}.
Finally, for concision, we will also write
\begin{equation}
\label{eq:signal}
\curr[\signal]
	= \vecfield(\curr) + \curr[\error]
\end{equation}
so $\curr[\signal]$ can be seen as a noisy \textendash\ and potentially biased \textendash\ estimator of $\vecfield(\curr)$ in \eqref{eq:RRM}.
Obviously, $\curr[\signal] = \vecfield(\curr)$ whenever $\curr[\bbound] = \curr[\sdev] = 0$;
we will refer to this case as ``deterministic''.

\subsection{Stochastic approximation}
\label{sec:cont2disc}

Our main goal in the sequel will be to connect the asymptotic behavior of the trajectories of \eqref{eq:RRM} to a continuous-time dynamical system on $\points$.
To do so, in analogy with the \acs{ODE} method for standard, Euclidean \acl{RM} schemes \citep{KY97,Ben99}, we will view \eqref{eq:RRM} as an inexact (forward) Euler discretization of the \acli{RMD}
\begin{equation}
\label{eq:RMD}
\tag{RMD}
\dotorbit{\ctime}
	= \vecfield(\orbit{\ctime})
\end{equation}
and we will try to establish a measure of ``closeness'' between the discrete-time sequence $\curr$ generated by \eqref{eq:RRM} and the continuous-time orbits $\orbit{\ctime}$ of \eqref{eq:RMD}.

In the Euclidean case,
this is provided by taking an affine interpolation $\apt{\ctime}$ of $\curr$ that agrees with $\curr$ at all instances of the ``effective time'' variable $\curr[\efftime] = \sum\nolimits_{\runalt=\start}^{\run-1} \iter[\step]$, viz.
\begin{equation}
\label{eq:interpolation-Eucl}
\apt{\ctime}
	= \curr
		+ \frac{\ctime - \curr[\efftime]}{\next[\efftime] - \curr[\efftime]}
			(\next - \curr)
	\quad
	\text{for all $\ctime \in [\curr[\efftime],\next[\efftime]]$, $\run=\running$}
\end{equation}
Of course, this definition is not meaningful in a Riemannian setting because of the lack of an affine structure on $\points$.
Instead, noting that the increments in \eqref{eq:interpolation-Eucl} can be rewritten as $(\next - \curr) / (\next[\efftime] - \curr[\efftime]) = \curr[\signal]$, a natural Riemannian analogue would be to follow the geodesic emanating from $\curr$ along $\curr[\signal]$ until reaching $\next$.
With this in mind, we define the \emph{geodesic interpolation} $\apt{\ctime}$ of $\curr$ as
\begin{equation}
\label{eq:interpolation}
\tag{GI}
\apt{\ctime}
	= \expof*{\curr}{(\ctime - \curr[\efftime]) \curr[\signal]}
	\quad
	\text{for all $\ctime \in [\curr[\efftime],\next[\efftime]]$, $\run\geq\start$},
\end{equation}
so, by construction,
\begin{enumerate*}
[(\itshape a\upshape)]
\item
$\apt{\curr[\efftime]} = \curr$ for all $\run$;
and
\item
each segment of $\apt{\ctime}$ is a geodesic.
\end{enumerate*}

Now, to compare $\apt{\ctime}$ to the solution orbits of \eqref{eq:RMD}, let $\flowmap\from\R_{+}\times\points\to\points$ denote the \emph{flow} of \eqref{eq:RMD}, \ie $\flow[h][\point]$ is simply the position at time $h\geq0$ of the solution orbit of \eqref{eq:RMD} that starts at $\point \in \points$ (recall here that $\vecfield$ is complete).
We then have the following notion of ``asymptotic closeness'':

\begin{definition}[\citefull{BH96}]
\label{def:RAPT}
We say that $\apt{\ctime}$ is an \acdef{APT} of the mean dynamics \eqref{eq:RMD} if, for all $\horizon > 0$, we have
\begin{equation}
\label{eq:RAPT}
\lim_{\ctime\to\infty}
	\sup_{0 \leq h \leq \horizon}
		\dist(\apt{\ctime + h} , \flow[h][\apt{\ctime}])
	= 0.
\end{equation}
\end{definition}

Intuitively, \cref{def:RAPT} states that one cannot distinguish between the geodesically interpolated process $\apt{\ctime + h}$ and the orbit of \eqref{eq:RMD} that starts at $\apt{\ctime}$ as $\ctime\to\infty$.
This is a highly non-trivial requirement, so much of the analysis to follow hinges on establishing exactly this property;
we carry all this out in \cref{sec:analysis}, where we also describe the precise connection between the asymptotic behavior of \eqref{eq:RRM} and that of \eqref{eq:RMD}.

\subsection{Technical assumptions}
\label{sec:assumptions}

We conclude this section with the basic assumptions that underlie the rest of our paper.
These are as follows:

\begin{assumption}
[\acl{RM} step-sizes]
\label{asm:step}
The step-size sequence $\curr[\step]$, $\run=\running$, of \eqref{eq:RRM} satisfies the \acl{RM} summability conditions
$\sum_{\run} \curr[\step] = \infty$
and
$\sum_{\run} \curr[\step]^2 < \infty$.
\end{assumption}

\begin{assumption}
[Error bounds]
\label{asm:errorbounds}
The bounds on the noise and bias in \eqref{eq:signal-stats} satisfy
\begin{equation}
\label{eq:errorbounds}
\sum_{\run=1}^{\infty} \curr[\step]^{2} \exof*{\curr[\sdev]^{2}}
	< \infty,
	\quad
	\sum_{\run=1}^{\infty} \curr[\step] \exof*{\curr[\bbound]}
	< \infty, 
	\quad
	\text{and}
	\quad
	\curr[\bbound] \to 0 \;\; \text{ \as.} 
\end{equation}
\end{assumption}

Variants of the above assumptions are standard in the context of (Euclidean) \acl{RM} methods, \cf \citet{KC78,Ben99}, and references therein.
Nonetheless, it is worth noting that \cref{asm:step} lies under the explicit control of the algorithm designer,
while \cref{asm:errorbounds} is \emph{implicit} and depends on the specific problem at hand \textendash\ the mechanism providing access to $\vecfield$, the specific form of \eqref{eq:RRM}, etc.
In this regard, \cref{asm:errorbounds} is more delicate than \cref{asm:step};
we discuss this issue in detail in \cref{sec:applications}, where we show that \cref{asm:errorbounds} is indeed satisfied for a wide range of practical algorithms that adhere to the general template \eqref{eq:RRM}.

Moving forward, to exclude cases where \eqref{eq:RRM} becomes unstable over time, a standard practice in the literature is to assume that the sequence $\curr$ is contained in a compact subset of $\points$, a property known as \emph{precompactness}.
Formally, we have:

\begin{assumption}
[Precompactness]
\label{asm:compact}
The set of iterates $\setdef{\curr}{\run=\running}$ has compact closure in $\points$.
\end{assumption}

Albeit standard in the literature \cite{KC78,BMP90,Ben99}, \cref{asm:compact} may be difficult to verify if $\points$ is not itself compact (\eg if $\points = \R^{\vdim}$).
To account for this, we introduce in \cref{sec:analysis} a set of structural hypotheses on $\points$ and $\vecfield$ which guarantee that \cref{asm:compact} holds \acl{wp1}.

With all this in hand, our last blanket assumption is a technical requirement that interfaces between the geometry of $\points$ and the dynamics of \eqref{eq:RMD}.
To state it, recall first that $\point\in \points$ is said to be \emph{conjugate} to $\pointalt =
\expof{\point}{\vvec}$ along the geodesic $\expof{\point}{\ctime\vvec}$ if the exponential map $\exp_{\point}(\cdot)$ is not a diffeomorphism in a neighborhood of $\vvec$ \citep{Lee97}.
In addition, define the \emph{Picard flow} $\Pflowmap\from\R_{+}\to\points$ associated with $\apt{\ctime}$ to be the dynamical system
\begin{equation}
\label{eq:Pflow}
\tag{PFlow}
\dot{\Pflowmap}(h)
     = \ptrans{\apt{\ctime + h}}{\Pflowmap(h)}{\vecfield(\apt{\ctime +h})}	
\end{equation}with initial condition ${\Pflowmap}(0) = \apt{\ctime }$.%
\footnote{The term ``Picard flow'' stems from the fact that the integral $\int_{0}^{h} \vecfield(\apt{\ctime+s}) \dd s$ is the basic iteration in Picard's method of successive approximations for solving an \acs{ODE} in Euclidean spaces.
In the case of \eqref{eq:Pflow}, the parallel transport is the extra ingredient required to express the idea of ``integrating $\vecfield$ along $\apt{\ctime}$'', so \eqref{eq:Pflow} can be seen as a bona fide generalization of the Picard iteration to Riemannian manifolds.}
Finally, given some $\horizon,\ctime > 0$, consider the sets
\begin{subequations}
\begin{align}
\conj_{\Pflowmap}(\ctime,\horizon)
	&\defeq \setdef*{h \in [0, \horizon ]}{\apt{\ctime + h} \textup{ is conjugate to } \Pflowmap(h)},
	\\
\conj_{\flowcurve}(\ctime,\horizon)
	&\defeq\setdef*{h \in [0, \horizon]}{\apt{\ctime + h} \textup{ is conjugate to } \flowmap_{h}(\apt{\ctime})},
\end{align}
\end{subequations}
and let $\conj(\ctime,\horizon) = \conj_{\Pflowmap}(\ctime,\horizon) \cup \conj_{\flowcurve}(\ctime,\horizon)$.
We then make the following assumption:

\begin{assumption}
[Nowhere dense conjugates]
\label{asm:conjugate}
The set $\conj(\ctime,\horizon)$ is nowhere dense in $[0, \horizon]$ for sufficiently large $t>0$ and for all $T>0$.
\end{assumption}

At first sight, \cref{asm:conjugate} may appear quite opaque but it is otherwise quite mild.
Indeed, since the set of points conjugate to $\apt{\ctime+h}$ is at most one-dimensional \citep{War65}, the only way that \cref{asm:conjugate} can fail is if the curves $\apt{\ctime+h}$ and $\flowcurve(h)$ or $\Pflowmap(h)$ simultaneously traverse the same one-dimensional submanifold of $\points$ \textendash\ a fact which occurs with probability $0$ if the distribution of $\curr[\error]$ is non-singular.
It is also worth noting that \cref{asm:conjugate} is automatically satisfied on negatively curved spaces by the Cartan–Hadamard theorem (\cf \cref{prop:bounded} in the next section) and, finally, it is straightforward to verify \cref{asm:conjugate} manually on many of the manifolds that arise in practical applications \textendash\ such as spheres, Grassmannians, Stiefel manifolds and fixed-rank spectrahedra, \cf \citep{Lee97,AMS08,massart2019curvature} and references therein.
We discuss all this in more detail in the next section.

\section{Analysis and main results}
\label{sec:analysis}

\subsection{Statement and discussion of main results}
\label{sec:statements}

The connecting tissue between \aclp{APT} and the mean dynamics \eqref{eq:RMD} is provided by the \acdef{ICT} sets of \eqref{eq:RMD}.
These are defined as follows:

\begin{definition}[\citefull{Ben99}]
\label{def:ICT}
Let $\sset$ be a nonempty compact subset of $\points$.
Then:
\begin{enumerate}
\item
$\sset$ is \emph{invariant} under \eqref{eq:RMD} if $\flow[\ctime][\sset] = \sset$ for all $\ctime\in\R$.
\item
\label{item:b}
$\sset$ is an \emph{attractor} of \eqref{eq:RMD} if it is invariant under \eqref{eq:RMD} and there exists a compact neighborhood $\cpt$ of $\sset$ such that $\lim_{\ctime\to\infty} \dist(\flow,\sset) = 0$ uniformly in $\point\in\cpt$.
\item
$\sset$ is \acdef{ICT} if it is invariant and $\flowmap\vert_{\sset}$ admits no attractors other than $\sset$.
\end{enumerate}
\end{definition}

In words, \ac{ICT} sets are the ``terminal objects'' of the dynamics \eqref{eq:RMD}:
the orbits that converge to an \ac{ICT} set cannot be contained in a smaller subset thereof, so \ac{ICT} sets may be viewed as minimal connected periodic orbits up to arbitrarily small numerical errors \citep[][Prop.~5.3]{Ben99}.
The importance of this property for the theory of stochastic approximation is owed to the following theorem:

\begin{theorem}[\citefull{BH96}]
\label{thm:ICT}
Let $\apt{\ctime}$, $\ctime\geq0$, be a precompact \acl{APT} of \eqref{eq:RMD}.
Then the limit set $\limset(\state) \defeq \intersect_{\ctime\geq0} \cl\setdef{\apt{\ctimealt}}{\ctimealt\geq\ctime}$ of $\state$ is an \ac{ICT} set of \eqref{eq:RMD}.
\end{theorem}

\begin{remark*}
\Cref{thm:ICT} above was originally stated in the context of abstract metric spaces;
we have adapted it here to the Riemannian setting of \eqref{eq:RMD} for concision and concreteness.
\endenv
\end{remark*}

The limit set theorem above provides a fundamental link between \acp{APT} and the long-run behavior of \eqref{eq:RMD} as captured by its \ac{ICT} sets.
That being said, the \ac{APT} property itself may be difficult to verify from first principles, so the application of \cref{thm:ICT} to \acl{RRM} algorithms can be just as difficult.
In the Euclidean case ($\points = \R^{\vdim}$), \citet{BH96} address this issue via a series of criteria under which standard (Euclidean) \acl{RM} methods give rise to an \ac{APT} of the associated mean dynamics \citep[][Props.~4.1 and 4.2]{Ben99}.
Unfortunately however, in a bona fide Riemannian setting, these criteria cannot be used because they are inextricably tied to the affine structure of $\R^{\vdim}$;
as a result, it is not clear how to leverage \cref{thm:ICT} to obtain a theory of stochastic approximation for \acl{RM} methods on Riemannian manifolds.
We tackle this question below: 

\begin{theorem}
\label{thm:APT}
Suppose that \crefrange{asm:step}{asm:conjugate} hold.
Then, \acl{wp1}, the geodesic interpolation $\apt{\ctime}$ of the sequence of iterates $\curr$ of \eqref{eq:RRM} is an \ac{APT} of \eqref{eq:RMD}.
\end{theorem}

\cref{thm:APT} plays a pivotal role in our paper and is the cornerstone for the analysis to follow.
In particular, by invoking \cref{thm:ICT,thm:APT} in tandem, we readily obtain the following characterization of the limiting behavior of \eqref{eq:RRM}:

\begin{corollary}
\label{cor:ICT}
Suppose that \crefrange{asm:step}{asm:conjugate} hold.
Then $\curr$ converges to an \ac{ICT} set of \eqref{eq:RMD} \acs{wp1}.
\end{corollary}

An important consequence of \cref{cor:ICT} is that the analysis of the stochastic, discrete-time system \eqref{eq:RRM} boils down to that of the \emph{deterministic}, \emph{continuous-time} system \eqref{eq:RMD}.
In this way, \cref{thm:APT} allows us to employ the same high-level strategies as the classic literature on stochastic approximation:
For example, if $\vecfield$ admits a potential or is geodesically (strictly) monotone, it is easy to verify that the only \ac{ICT} sets of \eqref{eq:RMD} are the roots of $\vecfield$ \citep{sternberg1999lectures}, so we readily recover the series of asymptotic convergence results mentioned in the introduction.
In \cref{sec:applications}, we present a wide range of problems whose \ac{ICT} sets coincide with the solutions of \eqref{eq:root}, and we further illustrate how \cref{thm:APT} captures a series of Riemannian stochastic approximation algorithms \textendash\ old and new \textendash\ in a unified fashion.

The proof of \cref{thm:APT} is fairly arduous, so we defer it to the end of this section;
instead, we proceed to discuss here in more detail the theorem's precompactness and conjugacy requirements (\cref{asm:compact,asm:conjugate} respectively).
First, as mentioned above, \cref{asm:compact} is common in the stochastic approximation literature but, in general, it is impractical to verify directly from the primitives of the problem at hand \textendash\ that is, the ambient manifold $\points$ and the defining vector field $\vecfield$.
Likewise, if the cut loci of different points on $\points$ happen to have a complicated topological structure, verifying \cref{asm:conjugate} could also be impractical.
To account for all this, we provide below a set of structural conditions on $\points$ and $\vecfield$ which guarantee that \cref{asm:compact,asm:conjugate} both hold \acl{wp1}:

\begin{proposition}
\label{prop:bounded}
Suppose that the following hypotheses hold:
\begin{enumerate}
[\upshape(H1)]
\item
\label[hypothesis]{hyp:Hadamard}
$\points$ is a Hadamard manifold with the Heine-Borel property.%
\footnote{A Hadamard manifold is a simply connected Riemannian manifold with non-positive sectional curvatures, and the Heine-Borel property simply posits that every closed bounded subset thereof is compact.}
\item
\label[hypothesis]{hyp:WC}
$\vecfield$ satisfies the weak coercivity condition
\begin{equation}
\label{eq:WC}
\tag{WC}
\inner*{\vecfield(\point)}{ \nabla \dist^{2}(\base,\point)}_{\point}
	\leq 0
\end{equation}
for some base point $\base\in\points$ and for all $\point\in\points \setminus \ball_{\coer}(\base)$ outside a closed geodesic ball $\ball_{\coer}(\base) \defeq \setdef{\point\in\points}{\dist(\base,\point) \leq \coer}$ of radius $\coer>0$ and centered at $\base$.
\end{enumerate}
Then, \acl{wp1}, \cref{asm:compact,asm:conjugate} hold as stated \textendash\ with \cref{asm:conjugate} only requiring \ref{hyp:Hadamard}.
\end{proposition}

\begin{corollary}
\label{cor:ICT-expl}
Suppose that \cref{asm:step,asm:errorbounds} and \cref{hyp:Hadamard,hyp:WC} hold.
Then $\curr$ converges to an \ac{ICT} set of \eqref{eq:RMD} \acs{wp1}.
\end{corollary}

To streamline our discussion, we postpone the proof of \cref{prop:bounded} to \cref{app:stability}.
For now, we simply note that the Hadamard requirement \ref{hyp:Hadamard} is fairly common in applications to Riemannian optimization and includes Stiefel and Grassmann manifolds (as homogeneous spaces), standard matrix manifolds, hyperbolic spaces, etc. \cite{BriHae99,HS20}.

In this regard, the most delicate requirement in \cref{prop:bounded} is \cref{hyp:WC}.
This hypothesis may be viewed as a Riemannian relaxation of the coercivity condition $\lim_{\point\to\infty} \inner{\vecfield(\point)}{\point}/\norm{\point}_{2} = -\infty$, which posits that the ``inward-pointing'' component of $\vecfield$ grows unbounded at infinity, a property which is frequently used to ensure the stability of Euclidean iterative algorithms \citep{Phe93,FP03}.
In our Riemannian setting, the role of the radial field is played by the gradient of the squared distance function $\nabla\dist^{2}(\base,\point)$, which is itself equal to $-2\log_{\point}(\base)$ if $\base$ lies in the injectivity radius of $\point$ (and hence for all $\point\in\points$ if $\points$ is Hadamard).
In addition, it is important to bear in mind that \eqref{eq:WC} \emph{does not} impose any growth requirements on the radial component of $\vecfield$;
it only requires that $\vecfield$ does not have a consistent outward-pointing component that could lead the process to diverge, so it is significantly lighter in that respect (hence the adjective ``weak'').

\subsection{Proof of \cref{thm:APT}}
\label{sec:proof}

Because the proof of \cref{thm:APT} and the geometric scaffolding required are quite delicate, we begin with a high-level outline outlining the main difficulties and technical challenges involved.
In brief, the basic obstacle that we have to overcome is as follows:
\begin{itemize}
\item
On the one hand, we need a coordinate system to compare and compute distances between different \emph{points} on $\points$.
This can be done efficiently in normal coordinates \citep{Lee97}.
\item
On the other hand, we also need to compare \emph{vectors} living on different tangent spaces.
In general, this comparison is very difficult to carry out in a normal coordinate frame, but it is much easier in the \acdef{PFS} that we describe in detail in \cref{app:pcoord}.
\end{itemize}

Intuitively, the normal coordinate system is where \emph{distances} behave as if the ambient space were Euclidean, and the \acl{PFS} is where \emph{vectors} behave as in the Euclidean setting.
Unfortunately however, the only regime where these two systems can coexist is when $\points$ is \emph{flat}, \ie the problem is ``essentially Euclidean'' to begin with.
To circumvent this obstacle, we take the following technical approach:

\begin{enumerate}
\item Based on the notion of \emph{Fermi coordinates} \citep{manasse1963fermi} \textendash\ which can be intuitively understood as ``normal coordinates along a curve'', \cf \cref{fig:Fermi} \textendash\ we begin by constructing an \emph{extended Fermi coordinate frame} that allows us to focus on a neighborhood of $\apt{\ctime}$ containing all the information we need.
[This is a challenging, but otherwise mostly technical, step that does not affect the big picture.]

\item
Using the extended Fermi coordinates constructed above, we can reduce the task of comparing the distance between two Riemannian curves to comparing several \emph{Euclidean}, albeit individually intractable, vector fields.
This step incurs an error term that is not present in the analysis of \emph{Euclidean} stochastic approximation schemes, and which is difficult to control in regions of high sectional curvature,
\cf \cref{eq:hold,eq:hold2}.

\item
To obtain expressions of vector fields that are more amenable to computation, we will switch from the extended Fermi coordinates to the \acl{PFS} and bound the difference between the two.
This step invariably introduces an additional error relative to the Euclidean analysis, \cf \cref{eq:f_to_p_error_vec,eq:f_to_p_error_pflow}.

\item
Serendipitously, these additional error terms can be managed without any further assumptions, and a series of arguments in the spirit of \citet{BH96} concludes our proof.
\end{enumerate}


\begin{figure}[tbp]
\centering
\includegraphics[width=0.9\textwidth]{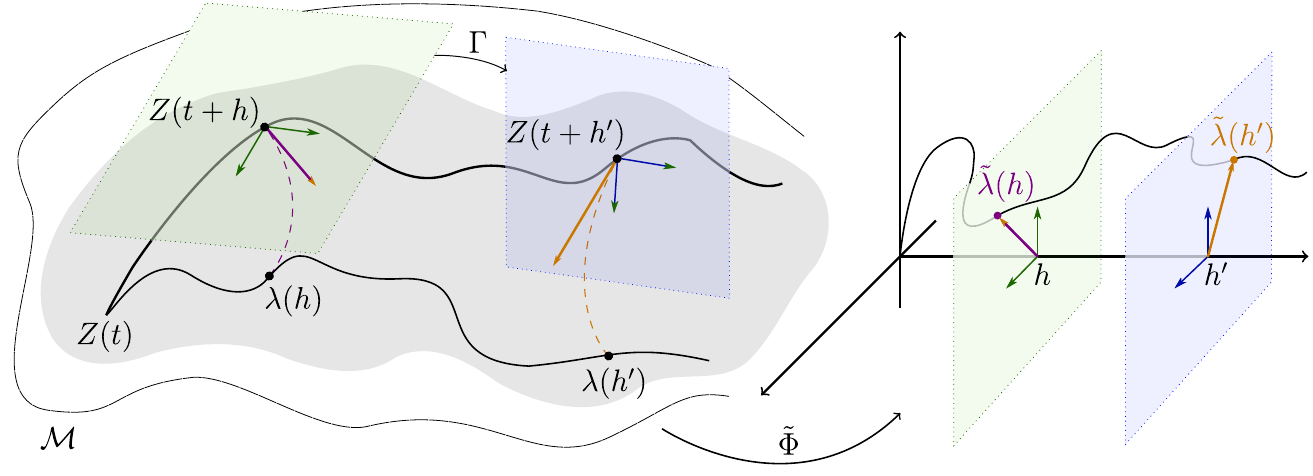}
\caption{Illustration of an extended Fermi coordinate frame.
The moving observer $\apt{\ctime+h}$ measures a curve $\Pflowmap(h)$ using time-indexed geodesics
and an ``inertial frame system'', \ie frames obtained via parallel transport from $\apt{\ctime}$.
For each $h$, $\normal{\Pflowmap}(h)\in \R^{\vdim}$ is the normal coordinate of $\Pflowmap(h)\in\points$.
The space-time map $\normal{\flowmap} \from \R_{+}\times\points\to\R_{+}\times\R^{\vdim}$ is locally defined on a neighborhood containing $\apt{\ctime+h}$ and $\Pflowmap(h)$.}
\label{fig:Fermi}
\vspace{-1ex}
\end{figure}


To formalize all this, we will need the following notions and results from Riemannian geometry:
\begin{itemize}
\item 
The concept of a \acli{PFS}.

\item 
A technical lemma by \citet{fujita1982onsager} and \citet{takahashi1981probability} which bounds the distortion of velocities measured by a moving observer on a manifold relative to flat spaces.

\item 
A comparison lemma to estimate the difference between parallel transport and the pushforward of the Riemannian exponential map.
\end{itemize} 
These elements are quite technical and involved, so we proceed directly to the proof of \cref{thm:APT} and we relegate all precise statements regarding the above to \cref{app:prelim}.

\para{Step 1: Discrete-to-continuous time comparisons}

Following \citet{Ben99}, consider the ``continuous-to-discrete'' counter
\begin{equation}
\label{eq:tinv}
\tinv(\ctime)
	= \sup\setdef{\run\geq\start}{\ctime\geq\curr[\efftime]}
\end{equation}
which measures the number of iterations required for the effective time $\curr[\efftime] = \sum_{\runalt=\start}^{\run-1} \iter[\step]$ to reach a given timestamp $\ctime \geq 0$.
Moroever, given an arbitrary sequence $\curr[A]$, we will denote its piecewise-constant interpolation as
\begin{equation}
\label{eq:piecewise-const}
\bar A(\ctime)
	= \curr[A]
	\quad
	\text{for all $\ctime \in [\curr[\efftime],\next[\efftime])$, $\run\geq\start$.}
\end{equation}
Using this notation, we may express the geodesic interpolation $\apt{\ctime}$ of $\curr$ in differential form as
\begin{align}
\label{eq:apt_shift}
\dot{\state}(t+h)
	&= \ptrans
		{\bar\state(\ctime+ h)}{\state(\ctime + h)}
		{\vecfield\left(\bar\state(\ctime+ h)\right) +  \bar{\error}(\ctime+ h)}
\end{align}
and we further let
\begin{equation}
\label{eq:bboundstar}
\bar{\step}^{\ast}(\ctime)
	\defeq \sup_{\ctime \leq h \leq \ctime + \horizon} \bar{\step}(h)
	\quad
	\text{and}
	\quad
\bar{\bbound}^{\ast}(\ctime)
	\defeq \sup_{\ctime \leq h \leq \ctime + \horizon} \bar{\bbound}(h)
	\quad
	\text{for all $\ctime,h\geq0$}.
\end{equation}
By the Stolz-Cesàro theorem, \cref{asm:step,asm:errorbounds} readily imply that $\lim_{\run\to\infty} \curr[\bbound] = 0$, so we also have $\lim_{\ctime\to\infty} \bar{\bbound}^{\ast}(\ctime) = 0$;
we will use this fact freely in the sequel.

As a last comparison step, we will also need a noise stability criterion.
To that end, let $\{\bvec_{\coord}(\run)\}_{\coord=1}^d$ be an arbitrary sequence of orthonormal bases for $\tspace[\curr]$, and let $\noise_{\run}^{\pcoord}$ be the (Euclidean) noise vector composed of components of the noise $\curr[\noise]$ in the basis $\{\bvec_{\coord}(\run)\}_{\coord=1\dotsc\vdim}$, viz.
\begin{equation}
\noise_{\coord,\run}^{\pcoord}
	\defeq \inner{\curr[\noise]}{\bvec_{\coord}(\run)}_{\curr}.
\end{equation}
It is then easy to see that $\exof{\noise_{\run}^{\pcoord} \vert \curr[\filter]} = 0$, and, moreover
\begin{equation}
\exof*{ \norm*{ \noise_{\run}^{\pcoord} }_{2}^{2} \given \curr[\filter] }
	=  \exof*{\norm{\curr[\noise]}^{2}_{\curr} \given \curr[\filter]}
	\leq \curr[\sdev^{2}]
\end{equation}
by \cref{asm:errorbounds}.
Then, letting
\begin{equation}
\label{eq:delta}
\Delta(\ctime;\horizon)
	\defeq \sup_{0\leq h\leq \horizon} \norm*{\int_{\ctime}^{\ctime+h} \bar\noise^{\pcoord}(s) \dd s}_{2},
\end{equation}
a classical argument by \citet[][\cf Eq.~(13) and onwards]{Ben99} readily gives
\begin{equation}
\label{eq:noise_stability}
\lim_{t\to \infty}\Delta(\ctime;\horizon)
	= 0
	\quad
	\text{for all $\horizon>0$ \as}.
\end{equation}

\para{Step 2: Preliminary error bounds}

We first note that, by \cref{asm:compact},
~$\sup_{\ctime} \radial\parens*{\apt{\ctime} } < \infty$ where $\radial(\cdot)\defeq \dist(\cdot,p)$ is the radial function defined in \eqref{eq:WC}.
We claim that \cref{asm:compact} also implies the boundedness of the Picard flow.
To see this, recall that the parallel transport is an isometry, so
\begin{align}
\norm*{\dot{\Pflowmap}(h)}_{\Pflowmap(h)}
	 = \norm*{\ptrans{\apt{\ctime + h}}{\Pflowmap(h)}{\vecfield(\apt{\ctime+h})}}_{\Pflowmap(h)}
	 = \norm*{\vecfield(\apt{\ctime+h})}_{\apt{\ctime+h}},
\end{align}
and hence
\begin{equation}
\sup_{0
	\leq h \leq \horizon}\dist\parens*{\Pflowmap(0), \Pflowmap(h)} \leq \horizon \cdot \sup_{0\leq h \leq \horizon}\norm*{\dot{\Pflowmap}(h)}_{\Pflowmap(h)}
	< \infty,
\end{equation}
which implies $\sup_{0\leq h \leq \horizon}\radial(\Pflowmap(h))<\infty$.
On the other hand, the boundedness for the flow follows readily from the fact that $\vecfield$ and $\dist$ are both $\lips$-Lipschitz,%
\footnote{Due to the smoothness of the flow, the function $\radial\parens*{\flowcurve(h)}$ is always differentiable in $h$ in the metric space sense \citep{ambrosio2005gradient}, even though $\flowcurve(h)$ might reach the cut locus of $\base$.}
so
\begin{align}
{\frac{\dd}{\dd h}\radial\parens*{\flowcurve(h)}}
	&\leq \norm*{\dot{\flowcurve}(h)}_{\flowcurve(h)}
		=  \norm*{\vecfield(\flowcurve(h))}_{\flowcurve(h)}
	\notag\\
	&\leq \norm*{\vecfield(\flowcurve(h))-\ptrans{ \base }{\flowcurve(h)}{\vecfield(\base)}}_{\flowcurve(h)} + \norm*{\ptrans{\base}{\flowcurve(h)}{\vecfield(\base)}}_{\flowcurve(h)}
	\notag\\
	&\leq \lips \radial\parens*{\flowcurve(h)} + \norm*{{\vecfield(\base)}}_{\base}.
\end{align}
An application of Grönwall's inequality then gives $\sup_{0\leq h \leq \horizon} \radial\parens*{\flowcurve(h)} < \infty$.
Therefore, all computations in the sequel can be restricted to a compact set, and
we may assume without loss of generality that $\bar{\step}^{\ast}(\ctime) < 1$ for all $\ctime\geq0$ and,
in addition, there exists some $\rbound \equiv \rbound({\horizon,\lips,\gbound})$ such that
\begin{align}
\label{eq:rbound}
\max\braces*{
		\dist\parens*{\apt{\ctime + h} , \flowcurve(h)}, \dist\parens*{\apt{\ctime + h} , \Pflowmap(h) }
		}
	\leq \rbound
	\quad
	\text{for all $h\in[0,T]$}.
\end{align}

\para{Step 3: Constructing the extended Fermi coordinates}

Let $\fnbhd$ be the neighborhood defined in \cref{app:fcoord} restricted to $[\ctime, \ctime+ \horizon]$, \ie $\fnbhd = \union_{h=0}^\horizon \nbhd_h$ with $\nbhd_h$ given by \eqref{eq:normal_nbhd}.
Clearly, $\fnbhd$ contains $\setdef{\apt{\ctime+h}}{h\in [0,\horizon]}$;
moreover, by construction, $\fdiffeo$ carries a system of orthonormal frames $\braces*{\bvec_\coord(h)}_{\coord=1}^\vdim$, one for each $\apt{\ctime+h}$.
In what follows, all quantities will be expressed in these frames unless explicitly mentioned otherwise.

Now, fix some $h_{0} \in [0,\horizon]$ and let $\curve_\flowcurve$ and $\curve_\Pflowmap$ be two minimizing geodesics such that $\curve_\flowcurve(0) = \curve_\Pflowmap(0) = \apt{\ctime+h_{0}}$, $\curve_\flowcurve(1) = \flowcurve(h_{0})$ and $\curve_\Pflowmap(1) = \Pflowmap(h_{0})$.
Our first goal will be to extend $\fnbhd$ to an open set of $\points$ that contains the geodesics $\curve_\flowcurve$ and $\curve_\Pflowmap$, while retaining the exponential mapping as a local diffeomorphism.
This will serve a dual purpose:
\begin{enumerate}
\item It enables us, for a fixed $h_{0}$, to consider the \aclp{PFS} at $\apt{\ctime+h_{0}}$ and $\flowcurve(h_{0})$ \textendash\ or $\apt{\ctime+h_{0}}$ and $\Pflowmap(h_{0})$, depending on the context \textendash\ so that we can easily compare the vector fields at these points;
see \cref{app:pcoord}.
\item We want to apply \cref{lem:fcoord} to the curves $\flowcurve(\cdot)$ and $\Pflowmap$;
however, for $\dfflow[h_{0}]$ and $\dfpflow[h_{0}]$ to make sense, $\fnbhd$ must contain both curves for at least an open time interval that includes $h_{0}$.
\end{enumerate}
This is where \cref{asm:conjugate} comes into play:
Since the conjugate points arise by definition when the exponential map ceases to be local diffeomorphisms \citep{Lee97}, it is reasonable to expect that, away from the time points where $\apt{\ctime+h}$ is conjugate to $\flowcurve({h})$ or $\Pflowmap({h})$, it is always possible to extend $\fnbhd$ to include the geodesics connecting $\apt{\ctime+h}$ to $\flowcurve({h})$ and $\Pflowmap({h})$.
In this regard, \cref{asm:conjugate} simply posits that there cannot be ``too many'' such conjugate points.

To formalize this, fix some $h_{0} \notin \conj(\ctime,\horizon)$ where $\conj(\ctime,\horizon)$ is defined as in \eqref{asm:conjugate}, and assume also that $t+ h_{0} \neq \curr[\efftime]$ for all $\run$ \textendash\ \ie $t+ h_{0}$ is not a ``kink point'' of \eqref{eq:interpolation}.
Since the exponential mapping is a local diffeomorphism away from conjugate points \citep{Lee97}, it follows that 
\(
\exp_{\apt{\ctime+h_{0}}} \from \tspace[\apt{\ctime+h_{0}}] \to \points
\)
is a local diffeomorphism at $\fpflow[h_{0}]$ and $\fflow[h_{0}]$, where $\fpflow[h_{0}]$ and $\fflow[h_{0}]$ are the normal coordinates of $\Pflowmap(h_{0})$ and $\flowcurve(h_{0})$ with center $\apt{\ctime+h_{0}}$ respectively.
By the continuity of the Picard flow and the frame system $\{\bvec_\coord(\cdot)\}_{\coord=1}^d$, there exists an open interval $(h_{\text{init}}, h_{\text{fin}})$ containing $h_{0}$ such that, for all $h \in (h_{\text{init}}, h_{\text{fin}})$,
\(
\exp_{\apt{\ctime+h}} \from \tspace[\apt{\ctime+h}] \to \points
\)
is a local diffeomorphism at $\fpflow$ and $\fflow$, where $\fpflow$ and $\fflow$ are the normal coordinates of $\Pflowmap({h})$ and $\flowcurve({h})$ with center $\apt{\ctime+h}$ respectively.

On that account, let $\curve^h_\flowcurve$ be a family of minimizing geodesics such that $\curve^h_\flowcurve(0) = \apt{\ctime+h}$ and $\curve^h_\flowcurve(1) = \flowcurve(h)$, and define $\curve^h_\Pflowmap$ similarly.
Since both $\curve^h_\flowcurve$ and $\curve^h_\Pflowmap$ are minimizing geodesics and $\flowcurve(h)$ and $\Pflowmap(h)$ are not conjugate to $\apt{\ctime+h}$, \citep[Theorem 10.15]{Lee97} ensures that no point on $\curve^h_\flowcurve$ or $\curve^h_\Pflowmap$ is conjugate to $\apt{\ctime+{h}}$.
Summing up, we have shown that the exponential mapping is a local diffeomorphism at any point of the set
\begin{equation}
\braces{\curve^h_\flowcurve}_{h\in (h_{\text{init}}, h_{\text{fin}})}
	\cup \braces{\curve^h_\Pflowmap}_{h\in (h_{\text{init}}, h_{\text{fin}})}.
\end{equation}

The final step in our construction is to consider the union of all such $(h_{\text{init}}, h_{\text{fin}})$ for all $h_{0} \notin \conj(\ctime,\horizon)$ and $t+ h_{0} \neq \curr[\efftime]$; we denote the set obtained in this way by $\nconj$.
More precisely, we claim that
\begin{enumerate*}
[\itshape a\upshape)]
\item
$\nconj$ is a dense open subset of $[0, \horizon]$;
and
\item
$\nconj$ can be written as a countable union of disjoint open intervals, \ie $\nconj = \union_{\runalt} (h_{\runalt}, h_{\runalt+1})$.
\end{enumerate*}
Indeed, the first claim follows readily from \cref{asm:conjugate} and the fact that the set $\braces{h_{0}: t+ h_{0} = \curr[\efftime] \text{ for some } \run}$ is countable.
The second claim is due to the compactness of $[0,T]$ and the fact that every nonempty open interval
in $\R$ contains a rational number.

In view of the above, it is possible to extend $\fnbhd$ and $\fdiffeo$ to an open set containing
\(
\braces{\curve^h_\flowcurve}_{h\in (h_{\text{\runalt}}, h_{\text{\runalt+1}})}
	\cup \braces{\curve^h_\Pflowmap}_{h\in (h_{\text{\runalt}}, h_{\text{\runalt+1}})},
\)
which, in turn, obviously contains $ \union_{h\in (h_{\text{\runalt}}, h_{\text{\runalt+1}})} \braces{\flowcurve(h), \Pflowmap(h)}$.
We call such a pair $(\fdiffeo,\fnbhd)$ the \emph{extended Fermi coordinate frame} because it not only contains the central curve $h\to \apt{\ctime +h}$ as in the classical case, but also $\flowcurve(h)$ and $\Pflowmap(h)$ for almost every $h\in[0,\horizon]$.

\para{Step 4: Controlling the distance by decomposition}

From this point forward, we will assume that all computations take place in the extended \acl{FCS} $(\fdiffeo,\fnbhd)$.
By the definition of $\flowcurve(\ctime)$, and given that $\fflow[{h}]$ is the normal coordinate of $\flowcurve({h})$ with center $\apt{\ctime + {h}}$, we have, for all $h\in \nconj$,
\begin{align}
\label{eq:decomp}
\dist(\apt{\ctime + h} , \flow[h][\apt{\ctime}]) &= \dist(\apt{\ctime + h} , \flowcurve(h))
	= \norm{\fflow[{h}]}_{2}
	\leq \norm{\fflow[{h}] - \fpflow}_{2}  + \norm{ \fpflow}_{2}.
\end{align}
Since $\nconj$ is a dense open subset of $[0,\horizon]$ and since it is a countable union of open intervals, it follows that $\fflow$ and $\fpflow$ are differentiable except on a set of measure zero.
We may thus write
\begin{align}
\norm{\fflow[{h}] - \fpflow}_{2}
	=  \norm*{\int_{0}^{h} \bracks[\big]{\dfflow[u] - \dfpflow[u]} \dd u}_{2}.
\end{align}
Moreover, by \cref{lem:fcoord} and the definition of \eqref{eq:RMD} and \eqref{eq:Pflow}, we have
\begin{subequations}
\begin{align}
\dfflowc
	&= \dot{\flowcurve}_{\coord}(u) - \dot{\state}_{\coord} (\ctime+u) + \bigoh\parens[\big]{\norm{\fflow[u]}_{2}^{2}}
= \fvecfield_{\coord}(u, \fflow[u]) - \dot{\state}_{\coord} (\ctime+u) + \bigoh\parens[\big]{\norm{\fflow[u]}_{2}^{2}},
	\\
\dfpflowc
	& = \dot{\Pflowmap}_{\coord}(u) - \dot{\state}_{\coord} (\ctime+u) + \bigoh\parens[\big]{\norm{\fpflow[u]}_{2}^{2}}
	= \cptransc - \dot{\state}_{\coord} (\ctime+u) + \bigoh\parens[\big]{\norm{\fpflow[u]}_{2}^{2}},
\end{align}
\end{subequations}
where
$\dot{\state}_{\coord}$ is defined in \cref{lem:fcoord}
and
$\fvecfield_{\coord}(u, \fflow[u])$ and $\cptransc$ are, respectively, the $\coord$-th components of the vectors $\vecfield\big(\flowcurve(u)\big)$ and $\ptrans{\apt{\ctime + u}}{\Pflowmap(u)}{\vecfield(\apt{\ctime+u})}$ in the frame induced by the normal coordinate with center $\apt{\ctime+u}$ and frame $\braces*{\bvec_\coord(u)}_{\coord=1}^\vdim$.
Now, denoting by $\fvecfield(u, \fflow[u])$ the (Euclidean) vector with components
$\fvecfield_{\coord}(u, \fflow[u])$ and, by $\cptrans$ the vector with components $\cptransc$, we may write
\begin{align}
\label{eq:hold}
\norm{\fflow[{h}] - \fpflow}_{2} \leq  \Bigg\Vert\int_{0}^{h} \left(\fvecfield(u, \fflow[u]) - \cptrans \right)\dd u\Bigg\Vert_{2}  + \int_{0}^{h} R_1(u) \dd u 
\end{align}where the remainder term $R_1(u)$ is of order $\bigoh\left( \norm{\fflow[u]}_{2}^{2} + \norm{\fpflow[u]}_{2}^{2} \right)$.
Noting that, by \eqref{eq:rbound},
\begin{subequations}
\begin{align}
\norm{\fflow[u]}_{2}
	&= \dist\left( \apt{\ctime+u}, \flow[u][\apt{t}] \right) \leq \rbound,
	\\
\norm{\fpflow[u]}_{2}
	&= \dist\left( \apt{\ctime+u}, \Pflowmap(\ctime+u) \right) \leq \rbound,
\end{align}
\end{subequations}
we have $\norm{\fflow[u]}_{2}^{2} +  \norm{\fpflow[u]}_{2} ^{2} \leq \rbound  \left( \norm{\fflow[u]}_{2} + \norm{\fpflow[u]}_{2}  \right)$, we get $R_1(u) = \bigoh_\rbound\left(\norm{\fflow[u]}_{2} + \norm{\fpflow[u]}_{2}  \right) $ where $\bigoh_\rbound(\cdot)$ includes any constants depending on $\rbound$.

In the same vein, denoting by $\dot{\normal{\state}} (\ctime+u)$ the Euclidean vector whose $\coord$-th component is $\dot{\state}_{\coord} (\ctime+u)$, we have
\begin{align}
\label{eq:hold2}
\norm{\fpflow}_{2} \leq  \Bigg\Vert\int_{0}^{h} \left( \cptrans - \dot{\normal{\state}} (\ctime+u)  \right)\dd u \Bigg\Vert_{2} + \int_{0}^{h} R_{2}(u) \dd u 
\end{align}
with $R_{2}(u) = \bigoh\left( \norm{\fpflow[u]}_{2}^{2} \right) = \bigoh_\rbound\left(\norm{\fpflow[u]}_{2}  \right) $.

\para{Step 5: From Fermi to parallel coordinates}

So far, we have reduced the proof to comparing the vectors in \cref{eq:hold,eq:hold2}.
However, these vectors are not amenable to further computation as they are expressed in the frames induced by the normal coordinates, and these frames may not even be orthonormal.

On the other hand, when expressed in the \acl{PFS} (see \cref{app:pcoord}) with a common base point $\apt{\ctime+u}$, the vectors $\vecfield\big(\flowcurve(u)\big) $ and $\ptrans{\apt{\ctime + u}}{\Pflowmap(u)}{\vecfield(\apt{\ctime+u})}$ possess some favorable properties.
To see this, parallel transport the frame $\braces*{\bvec_\coord(u)}_{\coord=1}^\vdim$ along the geodesic from $\apt{\ctime+u}$ to form an orthonormal frame $\braces*{\bvec'_\coord(u)}_{\coord=1}^\vdim$ of $\tspace[\Pflowmap(u)]$, and consider the Euclidean vector $\cptransp$  whose components are defined as 
\begin{align}
\label{eq:dotfpflow-compo}
\cptranscp
	&\defeq \inner{ \ptrans{\apt{\ctime + u}}{\Pflowmap(u)}{\vecfield(\apt{\ctime+u})} }{ \bvec'_\coord(u)}_{\Pflowmap(u)}
	= \inner{ \vecfield(\apt{\ctime+u}) }{ \bvec_\coord(u)}_{\apt{\ctime+u}}.
\end{align}
Similarly, parallel transport the frame $\braces*{\bvec_\coord(u)}_{\coord=1}^\vdim$ along the geodesic from $\apt{\ctime+u}$  to form an orthonormal frame $\braces*{\bvec''_\coord(u)}_{\coord=1}^\vdim$ of $\tspace[\flowcurve(u)]$, and let
\(
\pvecfield_{\coord}\left(u, \fflow[u]\right)
	\defeq \inner{\vecfield\big(\flowcurve(u)\big) }{ \bvec''_\coord(u)}_{\flowcurve(u)}.
\)
Since the parallel transport is an isometry, we get
\begin{align}
\pvecfield_{\coord}\left(u, \fflow[u]\right) - \cptranscp
	&=  \inner*{\vecfield\big(\flowcurve(u)\big) }{ \bvec''_\coord(u)}_{\flowcurve(u)}
		- \inner*{ \ptrans{\apt{\ctime + u}}{\Pflowmap(u)}{\vecfield(\apt{\ctime+u})} }{ \bvec'_\coord(u)}_{\Pflowmap(u)}
	\notag\\
	&= \inner*{\ptrans{\flowcurve(u)}{\apt{\ctime+u}}{\vecfield\big(\flowcurve(u)\big)} }{ {\bvec_\coord(u)}}_{\apt{\ctime+u}}
		- \inner*{{\vecfield(\apt{\ctime+u})} }{ \bvec_\coord(u)}_{\apt{\ctime+u}}
	\notag\\
	&= \inner*{\ptrans{\flowcurve(u)}{\apt{\ctime+u}}{\vecfield\big(\flowcurve(u)\big)} - \vecfield(\apt{\ctime+u})}{ {\bvec_\coord(u)}}_{\apt{\ctime+u}},
\end{align}
and hence
\begin{align}
\norm*{\pvecfield\left(u, \fflow[u]\right) - \cptransp}_{2}
	&=  \norm{\ptrans{\flowcurve(u)}{\apt{\ctime+u}}{\vecfield\big(\flowcurve(u)\big)}  - \vecfield(\apt{\ctime+u})}_{\apt{\ctime+u}}
	\notag\\
\label{eq:lip-PFS}
	&\leq \lips \dist(\flowcurve(u), \apt{\ctime+u}) = \lips \norm{\fflow[u] }_{2}
\end{align}
where we have used the fact that $\vecfield$ is $\lips$-Lipschitz and that $\fflow[u]$ are normal coordinates with center $\apt{\ctime+u}$.

We would therefore like to replace $\fvecfield\left(u, \fflow[u]\right) - \cptrans $ with $\pvecfield\left(u, \fflow[u]\right) - \cptransp$ in \eqref{eq:hold}.
To this end, consider the difference in the $\coord$-th component of $\fvecfield\left(u, \fflow[u]\right)$ and $\pvecfield\left(u, \fflow[u]\right)$
\begin{align}
{\pvecfield_{\coord}\left(u, \fflow[u]\right) - \fvecfield_{\coord}\left(u, \fflow[u]\right)}
	= \inner*{\vecfield\big(\flowcurve(u)\big) }{ \bvec''_\coord(u)}_{\flowcurve(u)}
		-\inner*{\vecfield\big(\flowcurve(u)\big) }{ \frac{\del}{\del \point_\coord}\Big\vert_{\flowcurve(u)}  }_{\flowcurve(u)}
\end{align}
where $\frac{\del}{\del \point_i}\Big\vert_{\flowcurve(u)}$ is the $\coord$-th basis in the frame induced by the normal coordinate with center $\apt{\ctime+u}$ and frame $\{\bvec_\coord(u)\}_{\coord=1}^d$.
More specifically, denote by $\vvec$ the vector $\sum_{\coord = 1}^\vdim \fflowc[u] \bvec_\coord(u) \in \tspace[\apt{\ctime+u}]$ and consider the family of geodesics 
\begin{equation}
\curve(s,s')
	\defeq \expof{\apt{t+u}}{s(\vvec + s'\bvec_\coord(u))}
\end{equation}
so
\begin{equation}
\frac{\del}{\del \point_i}\Big\vert_{\flowcurve(u)}
	= \frac{\del}{\del s'} \curve(1,0)
	= \dd\exp_{\apt{\ctime+u}}(\vvec)(\bvec_\coord(u)).
\end{equation}
Invoking Cauchy-Schwartz,
\cref{lem:dexp_vs_pt}, and the fact that $\{\bvec_\coord(u)\}_{\coord=1}^d$ is orthonormal, we get
\begin{align}
\abs*{\pvecfield_{\coord}\left(u, \fflow[u]\right) - \fvecfield_{\coord}\left(u, \fflow[u]\right)}
	&= \abs*{ \inner*{\vecfield\big(\flowcurve(u)\big) }{ \bvec''_\coord(u)- \frac{\del}{\del \point_\coord}\Big\vert_{ \flowcurve(u) }  }_{\flowcurve(u)} }
	\notag\\
	&\leq \norm{\vecfield\big(\flowcurve(u)\big)}_{\flowcurve(u)}
		\cdot \norm*{\bvec''_\coord(u)- \frac{\del}{\del\point_\coord}\Big\vert_{\flowcurve(u)}}_{\flowcurve(u)}
	\notag\\
	&\leq \gbound \cdot K_{\max} \cdot f_{-K_{\max}}( \vvec ) \cdot \norm{ \bvec_{\coord}(u)_{\perp} }_{\apt{\ctime+u}}
	\notag\\
\label{eq:hold3}
	&\leq \gbound \cdot K_{\max} \cdot f_{-K_{\max}}(\vvec)
		= \gbound \cdot K_{\max} \cdot f_{-K_{\max}}(\norm{\fflow[u]}_{2})
\end{align}
where $K_{\max} = \max\braces{\abs{K_\textup{up}},\abs{K_\textup{low}}}$.
Since $(\sinh x)/x-1 \leq \cosh x \leq e^{x}$ for all $x \geq 0$, \eqref{eq:fdef} and \eqref{eq:fKbound} yield
\begin{equation}
\label{eq:hold4}
f_{-K_{\max}}( \norm{\fflow[u]}_{2})
	\leq \frac{1}{K_{\max}} \exp\left(\sqrt{K_{\max}}\norm{\fflow[u]}_{2}\right)
	\leq  \frac{ e^{\rbound\sqrt{K_{\max}}} }{\rbound \cdot K_{\max}} \cdot \norm{\fflow[u]}_{2}
\end{equation}
where the last inequality follows from $\norm{\fflow[u]}_{2} = \dist(\apt{\ctime+u}, \flow[u][\apt{\ctime}]) \leq \rbound$.
Combining \eqref{eq:hold3} and \eqref{eq:hold4}, we thus get
\begin{align}
\abs*{\pvecfield_{\coord}\left(u, \fflow[u]\right) - \fvecfield_{\coord}\left(u, \fflow[u]\right)}
	&\leq \frac{\gbound e^{\rbound\sqrt{K_{\max}}} }{\rbound} \cdot \norm{\fflow[u]}_{2}.
\end{align}
In short, we have shown that
\begin{equation}
\label{eq:f_to_p_error_vec}
\fvecfield\left(u, \fflow[u]\right)  = \pvecfield\left(u, \fflow[u]\right) + R_3(u)
\end{equation}
where $R_3(u) = \bigoh_{\gbound, K_{\max}, \rbound}\left(\norm{\fflow[u]}_{2} \right)$ collects constants that depend on $\gbound$, $K_{\max}$ and $\rbound$.
Exactly the same computation shows that, for some $R_4(u) = \bigoh_{\gbound, K_{\max}, \rbound}\left(\norm{  \fpflow[u]}_{2} \right)$,
\begin{equation}
\label{eq:f_to_p_error_pflow}
\cptrans
	= \cptransp + R_4(u).
\end{equation}

\para{Step 6: Putting everything together}

With all these preliminaries in hand, we are finally in a position to complete our proof.

\begin{proof}[Proof of \cref{thm:APT}]
We will proceed by bounding \eqref{eq:hold} and \eqref{eq:hold2}.
Using \eqref{eq:f_to_p_error_vec}, \eqref{eq:f_to_p_error_pflow}, and \eqref{eq:lip-PFS} in \eqref{eq:hold}, we obtain:
\begin{align}
\nonumber
\norm{\fflow[{h}] - \fpflow}_{2} &\leq  \Bigg\Vert \int_{0}^{h} \left(\fvecfield\left(u, \fflow[u]\right) - \cptrans \right)\dd u\Bigg\Vert_{2}  + \int_{0}^{h} R_1(u) \dd u  \\
\nonumber
&\leq \int_{0}^{h} \norm*{\fvecfield\left(u, \fflow[u]\right) - \cptrans}_{2}\dd u + \int_{0}^{h} R_1(u) \dd u \\
\nonumber
&\leq \int_{0}^{h} \norm*{\pvecfield\left(u, \fflow[u]\right) - \cptransp}_{2}\dd u + \int_{0}^{h} (R_1+R_3+R_4)(u) \dd u \\
\label{eq:hold5}
&\leq \lips \int_{0}^{h} \norm*{ \fflow[u]}_{2}\dd u+ \int_{0}^{h} (R_1+R_3+R_4)(u) \dd u.
\end{align}

We next turn to \eqref{eq:hold2}.
Our first task is to obtain an expression for $\dot{\normal{\state}} (\ctime+u)$, \ie the Euclidean vector whose $\coord$-th component is $\dot{{\state}}_{\coord}(\ctime+u)$.
To this end, fix an iteration count $\run$, and consider all $u$ such that $\ctime + u \in [\efftime_{\run}, \efftime_{\run+1})$ (that is, consider only the interpolation between $\curr$ and $\next$).
We claim that $\dot{{\state}}_{\coord}(\ctime+u)$ is constant throughout all such $u$, and, in particular, $\dot{{\state}}_{\coord}(\ctime+u) = \dot{{\state}}_{\coord}(\efftime_{\run})$.
This readily follows by noticing that
\begin{enumerate}
\item
The curve $\state(\ctime)$ is a geodesic segment when restricted to $[\curr[\efftime], \next[\efftime])$; see \eqref{eq:interpolation}.
\item
The Fermi coordinates along $\state(\cdot)$, when restricted to $\setdef{\state(s)}{ s \in [\curr[\efftime], \next[\efftime])}$, is simply a \acl{PFS};
this follows from the fact that the frame $\{\bvec_\coord(u)\}_{\coord=1}^d$ is obtained from parallel transporting $\{\bvec_\coord(\curr[\efftime])\}_{\coord=1}^d$ along $\apt{\cdot}$ for all $u$ such that $\ctime+u\in [\curr[\efftime], \next[\efftime])$.
\end{enumerate}
In this way, a simple calculation akin to \eqref{eq:comp_in_pcoord} yields, for all such $u$,
\begin{align}
\nonumber
\dot{{\state}}_{\coord}(\ctime+u) &= \inner{\dot{\state}(\ctime+u)}{\bvec_i(u)}_{\apt{\ctime+u}} \\
\nonumber
&= \inner*{ \ptrans{\apt{\ctime+u}}{\apt{\curr[\efftime]}}{\left(\dot{\state}(\ctime+u) \right)}}{\bvec_i(\curr[\efftime])}_{\apt{\curr[\efftime]}} \\
\label{eq:hold6}
&= \inner*{ \vecfield(\apt{\curr[\efftime]}) + \curr[\error] }{\bvec_i(\curr[\efftime])}_{\apt{\curr[\efftime]}}  
= \inner*{ \vecfield(\apt{\curr[\efftime]}) + \curr[\noise] + \curr[\bias] }{\bvec_i(\curr[\efftime])}_{\apt{\curr[\efftime]}} 
\end{align}where we have used \eqref{eq:apt_shift} and the definition of $\curr[\error]$ in the last equality.

Armed with the above, we can obtain a succinct expression for $\dot{\normal{\state}} (\ctime+u)$ as follows.
First, let $\faptbar$ be the normal coordinate of $\bar\state(\ctime+u)$ with center
$\apt{\ctime+u}$ (\ie $(u, \faptbar) = \fdiffeo(u, \bar\state(\ctime+u))$), and define an Euclidean vector $\pvecfield(u, \faptbar)$ by setting its $\coord$-th component to
\begin{equation}
\label{eq:pvec-bar-compo}
\pvecfield_{\coord}\left(u, \faptbar\right) \defeq \inner*{ \vecfield\left(\bar\state (\ctime + u) \right) }{\bvec_i(\tinv(\ctime+u))}_{ \bar\state (\ctime + u) }  
\end{equation}where the mapping $\tinv(\cdot)$ is defined in \eqref{eq:tinv}.
Define also the Euclidean noise and bias vectors $\noise_{\run}^{\pcoord}$ and $\bias_{\run}^{\pcoord}$ by setting their components to
\begin{subequations}
\begin{align}
\noise_{\coord,\run}^{\pcoord}
	&\defeq \inner*{  \curr[\noise] }{ \bvec_i(\tinv(\ctime+u)) }_{ \bar\state (\ctime + u) },
	\\
\bias_{\coord,\run}^{\pcoord}
	&\defeq\inner*{  \curr[\bias] }{  \bvec_i(\tinv(\ctime+u)) }_{ \bar\state (\ctime + u) }.
\end{align}
\end{subequations}
Then \eqref{eq:hold6} states precisely that
\begin{align}
\label{eq:dot-apt-formula}
\dot{\normal{\state}} (\ctime+u) = \pvecfield(u, \faptbar) + \bar{\noise}^{\pcoord}(\ctime+u) + \bar{\bias}^{\pcoord}(\ctime+u).
\end{align}
Substituting \eqref{eq:dot-apt-formula} into \eqref{eq:hold2} and invoking \eqref{eq:f_to_p_error_pflow}, \eqref{eq:bboundstar}, and \eqref{eq:delta}, we obtain 
\begin{align}
\norm{\fpflow}_{2}
	&\leq \norm*{
		\int_{0}^{h} \left( \cptrans - \pvecfield(u, \faptbar) - \bar{\noise}^{\pcoord}(\ctime+u) - \bar{\bias}^{\pcoord}(\ctime+u)  \right)\dd u
		}_{2}
		+ \int_{0}^{h} R_{2}(u) \dd u
	\notag\\
	&\leq \int_{0}^{h} \norm*{\cptransp - \pvecfield(u, \faptbar) }_{2} \dd u
	\notag\\
	&\qquad
		+ \norm*{ \int_{0}^{h} \bar{\noise}^{\pcoord}(\ctime+u) \dd u}_{2} + \norm*{ \int_{0}^{h} \bar{\bias}^{\pcoord}(\ctime+u) \dd u}_{2} + \int_{0}^{h} \bracks{R_{2}(u) + R_4(u)} \dd u
	\notag\\
	&\leq \int_{0}^{h} \norm*{\cptransp - \pvecfield(u, \faptbar) }_{2} \dd u
	\notag\\
	&\qquad
		+ \Delta(\ctime,\horizon) + \bar{\bbound}^{\ast}(\ctime) \cdot h  + \int_{0}^{h} \bracks{R_{2}(u) + R_4(u)} \dd u.
\label{eq:hold7}
\end{align}
To bound the first term in \eqref{eq:hold7}, recall \eqref{eq:dotfpflow-compo} and \eqref{eq:pvec-bar-compo}.
An identical argument leading to \eqref{eq:vec_in_pcoord} shows that
\begin{align}
\norm*{\cptransp - \pvecfield(u,\faptbar) }_{2}
	&= \norm*{\vecfield(\state(\ctime+u)) - \ptrans{\bar\state(\ctime+u)}{\state(\ctime+u)}{\vecfield(\bar\state(\ctime+u))}}_{\state(\ctime+u)}
	\notag\\
\label{eq:hol8}
	&\leq \lips \cdot\dist\left(\bar\state(\ctime+u), \state(\ctime+u) \right).
\end{align}
Since $\apt{\cdot}$ is a (not necessarily minimizing) geodesic on $[\tinv(t+u), \ctime+u]$, \cref{eq:dot-apt-formula,eq:bboundstar} yield
\begin{align}
\dist(\bar\state(\ctime+u),\state(\ctime+u))
	&\leq  \norm*{ \int_{\tinv(t+u)}^{t+u} \dot{\normal{\state}}(\ctimealt) \dd \ctimealt }_{2}
	\notag\\
	&= \norm*{  \int_{\tinv(t+u)}^{t+u}   \pvecfield(\ctimealt, \faptbar[\ctimealt]) + \bar{\noise}^{\pcoord}(\ctimealt) + \bar{\bias}^{\pcoord}(\ctimealt) \dd \ctimealt }_{2}
	\notag\\
	&\leq \norm*{  \int_{\tinv(t+u)}^{t+u}   \pvecfield(\ctimealt, \faptbar[\ctimealt]) \dd \ctimealt}_{2}
	\notag\\
	&\qquad
		+ \norm*{\int_{\tinv(t+u)}^{t+u} \bar{\noise}^{\pcoord}(\ctimealt)\dd \ctimealt}_{2}
		+ \norm*{ \int_{\tinv(t+u)}^{t+u} \bar{\bias}^{\pcoord}(\ctimealt) \dd \ctimealt }_{2}
	\notag\\
	&\leq \left(\gbound + \bar{\bbound}^{\ast}(\ctime) \right)\cdot \left(\int_{\tinv(\ctime+u)}^{\ctime+u} \dd \ctimealt \right)
		+ \norm*{ \int_{\tinv(t+u)}^{t+u} \bar{\noise}^{\pcoord}(\ctimealt)\dd \ctimealt}_{2}
	\notag\\
\label{eq:hold9}
	&\leq \left(\gbound + \bar{\bbound}^{\ast}(\ctime) \right)\bar{\step}^{\ast}(\ctime)
		+ \norm*{ \int_{\tinv(\ctime+u)}^{\ctime+u} \bar{\noise}^{\pcoord}(\ctimealt)\dd \ctimealt}_{2}.
\end{align}
For $\ctime$ large enough, we have $\bar{\step}^{\ast}(t)  < 1$, and hence 
\begin{align}
\norm*{ \int_{\tinv(t+u)}^{t+u} \bar{\noise}^{\pcoord}(\ctimealt)\dd \ctimealt}_{2}
	&\leq \norm*{ \int_{ \ctime-1}^{ \tinv(\ctime+u) } \bar{\noise}^{\pcoord}(\ctimealt)\dd \ctimealt}_{2} 		+ \norm*{ \int_{\ctime-1}^{\ctime+u} \bar{\noise}^{\pcoord}(\ctimealt)\dd \ctimealt}_{2}
	\leq 2\Delta(\ctime-1,\horizon+1).
\label{eq:hold10}
\end{align}
Combining \crefrange{eq:hold7}{eq:hold10} then gives
\begin{align}
\norm{\fpflow}_{2}
	&\leq \lips h
		\bracks{\parens{\gbound + \bar{\bbound}^{\ast}(\ctime)}\bar{\step}^{\ast}(\ctime)
	\notag\\
	&\qquad
		+ 2\Delta(\ctime-1,\horizon+1)}
		+ \Delta(\ctime,\horizon)
		+ h \bar{\bbound}^{\ast}(\ctime)
		+ \int_{0}^{h} \bracks{R_{2}(u) + R_4(u)} \dd u
	\notag\\
\label{eq:hold11}
	&\leq h C_{\lips, \gbound} \left(   \bar{\bbound}^{\ast}(\ctime) + \bar{\step}^{\ast}(\ctime) +  \Delta(\ctime-1,\horizon+1) \right) + \int_{0}^{h} \bracks{R_{2}(u) + R_4(u)} \dd u
\end{align}for some constant $ C_{\lips, \gbound}$ that depends only on $\lips$ and $\gbound$.
Using \eqref{eq:hold5} and \eqref{eq:hold11}, we can then bound \eqref{eq:decomp} as
\begin{align}
\dist(\apt{\ctime + h} , \flow[h][\apt{\ctime}])
	&= \norm{\fflow[{h}]}_{2}
		\leq  \norm{\fflow[{h}]}_{2}  + \norm{ \fpflow}_{2}
	\label{eq:hold12}
	\\
	&\leq  \norm{\fflow[{h}] - \fpflow}_{2}  + 2\norm{ \fpflow}_{2}
	\notag\\
	&\leq \lips \int_{0}^{h} \norm*{ \fflow[u] }_{2}\dd u  + 2 h C_{\lips, \gbound} \left(   \bar{\bbound}^{\ast}(\ctime) + \bar{\step}^{\ast}(\ctime) +  \Delta(\ctime-1,\horizon+1) \right)
	\notag\\ 
	&+ 3\int_{0}^{h} \bracks{R_1(u) + R_{2}(u) + R_3(u) + R_4(u)} \dd u
	\label{eq:hold13}
\end{align}
where $\bracks{R_1(u) + R_{2}(u) + R_3(u) + R_4(u)} = \bigoh_{\gbound,K_{\max},\rbound} \left( \norm*{\fflow}_{2} + \norm*{\fpflow}_{2}  \right)$.
Therefore, there exists a constant $C_{L,\gbound,K_{\max},\rbound}$ depending only on $L,\gbound,K_{\max}$, and $\rbound$ such that we may bound \eqref{eq:hold12} as
\begin{align}
\norm{\fflow[{h}]}_{2}  + \norm{ \fpflow}_{2}
	&\leq C_{L,\gbound,K_{\max},\rbound} \int_{0}^{h} \left( \norm*{\fflow}_{2} + \norm*{\fpflow}_{2}  \right) \dd u
	\notag\\
	&+ 2 h C_{\lips, \gbound} \left(   \bar{\bbound}^{\ast}(\ctime) + \bar{\step}^{\ast}(\ctime) +  \Delta(\ctime-1,\horizon+1) \right).
\end{align}
Grönwall's inequality then implies 
\begin{align}
\label{eq:hold14}
\norm{\fflow[{h}]}_{2}  + \norm{ \fpflow}_{2}
	&\leq 2 h C_{\lips, \gbound} \left(   \bar{\bbound}^{\ast}(\ctime) + \bar{\step}^{\ast}(\ctime) +  \Delta(\ctime-1,\horizon+1) \right) e^{h\cdot C_{L,\gbound,K_{\max},\rbound}}.
\end{align}
From \eqref{eq:hold14}, we may conclude that, \acl{wp1}, we have
\begin{align}
\lim_{\ctime \to \infty} \sup_{\ctime \leq h \leq \horizon}
	&\dist(\apt{\ctime + h} , \flow[h][\apt{\ctime}])
	\notag\\
	&\leq \lim_{\ctime \to \infty} \sup_{\ctime \leq h \leq \horizon } \left(\norm{\fflow[{h}]}_{2}  + \norm{ \fpflow}_{2} \right)
	\notag\\
	&\leq \lim_{\ctime \to \infty}\horizon\cdot 2C_{\lips, \gbound} \left(   \bar{\bbound}^{\ast}(\ctime) + \bar{\step}^{\ast}(\ctime) +  \Delta(\ctime-1,\horizon+1) \right) \cdot e^{\horizon\cdot C_{L,\gbound,K_{\max},\rbound}}
	= 0
\end{align}
since $\lim_{\ctime \to \infty}  \bar{\bbound}^{\ast}(\ctime) = \lim_{\ctime \to \infty}  \bar{\step}^{\ast}(\ctime) = 0$ by assumption, and $\lim_{\ctime \to \infty}  \Delta\left(\ctime-1,\horizon+1 \right) =0$ by \eqref{eq:noise_stability}, both \acl{wp1}.
\end{proof}

\section{Applications and implications}
\label{sec:applications}

Taking a step back, our goal in this section will be to show how a diverse range of algorithms can be seen as special cases of \eqref{eq:RRM}, enabling in this way the use of \cref{thm:APT,cor:ICT,cor:ICT-expl} to deduce their convergence properties.
To simplify our presentation, we will make the following technical assumption:

\begin{assumption}
\label[assumption]{asm:injectivity}
The injectivity radius of $\points$ is bounded from below by $\delta >0$.
\end{assumption}

\begin{remark*}
Recall that the injectivity radius of $\points$ at $\base$ is the radius of the largest geodesic ball on which $\exp_{\base}$ is a diffeomorphism, and the injectivity radius of $\points$ is the infimum over all such radii \citep{Lee97}.
In this regard, \cref{asm:injectivity} serves to ensure that the exponential map is invertible at consecutive iterates of \eqref{eq:RRM} so no local topological complications can arise;
we will in fact prove in \cref{prop:algorithms} that $\log_{\curr} \defeq \exp_{\curr}^{-1}$ is well-defined for all sufficiently large $\run$ in the algorithms to follow.
\endenv
\end{remark*}

For concreteness, we will also assume that the algorithms considered in this section have black-box access to $\vecfield$ via a \acdef{SFO} \citep{Nes04}.
Specifically, when called at $\point\in\points$ with random seed $\seed\in\seeds$, an \ac{SFO} returns a random vector $\orcl(\point;\seed) \in \tspace$ of the form
\begin{equation}
\label{eq:SFO}
\tag{SFO}
\orcl(\point;\seed)
	= \vecfield(\point)
		+ \err(\point;\seed)
\end{equation}
where the error term $\err(\point;\seed)\in \tspace$ is zero-mean and bounded in $L^{2}$,
\ie
\begin{equation}
\label{eq:err}
\exof{\err(\point;\seed)}
	= 0 
	\quad
	\text{and}
	\quad
\exof{\norm{\err(\point;\seed)}^{2}_\point}
	\leq \sdev^{2}
	\quad
	\text{for all $\point\in\points$.}
\end{equation}
Finally, we will also assume that all algorithms under study are run with a step-size policy $\curr[\step]$ such that
\begin{equation}
\label{eq:step}
\frac{A}{\run}
	\leq \curr[\step]
	\leq \frac{B}{\run^{1/2} (\log\run)^{1/2+\eps}}
	\quad
	\text{for some $A,B,\eps>0$ and all $\run=\running$}
\end{equation}
Clearly, \cref{asm:step} is satisfied automatically under \eqref{eq:step}.

\begin{remark*}
To facilitate comparisons with the Riemannian optimization literature, we will sometimes refer to queries of the \ac{SFO} $\orcl(\point;\seed)$ as (stochastic) ``gradients'';
we stress however that, in general, \emph{$\vecfield$ is not assumed to admit a potential}.
\endenv
\end{remark*}

\subsection{Basic examples}
\label{sec:examples}

We now proceed to outline below a \textendash\ highly incomplete \textendash\ series of algorithms that can be seen as special cases of the general template \eqref{eq:RRM}.

\begin{algo}
[\Aclp{RSGM}]
\label{alg:RSGM}
The simplest \acdef{RSGM} \citep{Bon13} queries \eqref{eq:SFO} and proceeds as
\begin{equation}
\label{eq:RSGM}
\tag{RSGM}
\next
	= \exp_{\ssstyle\curr} \left( \curr[\step] \orcl(\curr;\curr[\seed]) \right),
\end{equation}
As such, \eqref{eq:RSGM} is an \ac{RRM} scheme with $\curr[\noise] = \err(\curr;\curr[\seed])$ and $\curr[\bias] = 0$.
\endenv
\end{algo}

\begin{algo}
[\Aclp{RPPM}]
\label{alg:RPPM}
The \textpar{deterministic} \acdef{RPPM} \cite{ferreira2002proximal} is an implicit update rule of the form
\begin{equation}
\label{eq:RPPM}
\tag{RPPM}
\exp^{-1}_{\next} (\curr) = -\curr[\step] \vecfield(\next).
\end{equation}
The \ac{RRM} representation of \eqref{eq:RPPM} is then obtained by taking
\(
\curr[\bias] = \ptrans{\next}{\curr}{\vecfield(\next)} - \vecfield(\curr)
\)
and
$\curr[\noise] = 0$
in the decomposition \eqref{eq:error} of the error term $\curr[\error]$ of \eqref{eq:RRM}.
\endenv
\end{algo}

\begin{algo}[\Acl{RSEG}]
\label{alg:RSEG}
Inspired by the original work of \citet{Kor76}, the \acdef{RSEG} method \cite{tang2012korpelevich, neto2016extragradient} proceeds as
\begin{equation}
\label{eq:RSEG}
\tag{RSEG}
\begin{alignedat}{2}
\lead
	&= \expof{\ssstyle\curr}{\curr[\step]  \orcl(\curr;\curr[\seed])},
	\\
\next
	&= \expof{\ssstyle\curr}{\ptrans{\lead}{\curr}{\curr[\step]\orcl(\lead;\lead[\seed])}}
\end{alignedat}
\end{equation}
where $\curr[\seed]$ and $\lead[\seed]$ are independent seeds for \eqref{eq:SFO}.
Thus, to recast \eqref{eq:RSEG} in the \ac{RRM} framework, simply take $\curr[\noise] = \ptrans{\ssstyle\lead}{\ssstyle\curr}{ \err(\lead;\lead[\seed])}$ and $\curr[\bias] = \ptrans{\ssstyle\lead}{\ssstyle\curr}{   \vecfield(\lead)}  - \vecfield(\curr)$.
\endenv
\end{algo}

\begin{algo}[\Acl{ROG}]
\label{alg:ROG}
Compared to \eqref{eq:RSGM}, the scheme \eqref{eq:RSEG} involves two oracle queries per iteration.
Building on an original idea by \citet{Pop80}, the last oracle query can be ``recycled'', leading to the \acl{ROG} method
\begin{equation}
\label{eq:ROG}
\tag{ROG}
\begin{alignedat}{2}
\lead
	&= \expof{\ssstyle\curr}{\curr[\step] \orcl(\beforelead;{\prev[\seed]})},
	\\
\next
	&= \expof{\ssstyle\curr}{\ptrans{\lead}{\curr}{\curr[\step] \orcl(\lead;\curr[\seed])}}.
\end{alignedat}
\end{equation}
With this in mind, \eqref{eq:ROG} may be seen as a special case of \eqref{eq:RRM} by taking $\curr[\noise] = \ptrans{\lead}{\curr}{ \err(\lead;\curr[\seed])}$ and $\curr[\bias] = \ptrans{\lead}{\curr}{\vecfield(\lead)}-\vecfield(\curr)$.
\endenv
\end{algo}

In view of \cref{thm:APT,thm:ICT}, the convergence analysis of \crefrange{alg:RSGM}{alg:ROG} essentially boils down to verifying \crefrange{asm:step}{asm:conjugate}.
\Cref{prop:algorithms} below does much of the heavy lifting for this:

\begin{proposition}
\label{prop:algorithms}
Suppose that \cref{asm:injectivity} holds.
If \crefrange{alg:RSGM}{alg:ROG} are run with oracle and step-size parameters satisfying \eqref{eq:err} and \eqref{eq:step}, then:
\begin{enumerate}
\item
\Cref{asm:errorbounds} holds as stated.
\item
\Acl{wp1}, $\next$ lies in the injectivity radius of $\curr$ for all sufficiently large $\run$, so \cref{asm:conjugate} also holds as stated.
\end{enumerate}
\end{proposition}

\begin{corollary}
\label{cor:algorithms}
Suppose that \cref{hyp:Hadamard,hyp:WC} hold.
If \crefrange{alg:RSGM}{alg:ROG} are run with oracle and step-size parameters satisfying \eqref{eq:err} and \eqref{eq:step},
then, \acl{wp1}, the generated sequence $\curr$, $\run=\running$,
converges to an \ac{ICT} set of \eqref{eq:RMD}.
\end{corollary}

To streamline our discussion, we postpone the proof of \cref{prop:algorithms} to \cref{app:applications}.

\subsection{Algorithmic variants and modifications}
\label{sec:variants}

To increase the computational efficiency of Riemannian iterative schemes, several important operations are routinely tacked on to the base algorithms described in the previous section.
In view of this, we proceed below to illustrate a range of algorithmic variants and modifications that can be readily incorporated in the general framework of \eqref{eq:RRM}.

\para{Retraction-based methods}

When the exponential map is expensive to compute, a popular alternative is to employ a so-called \emph{retraction map} \citep{AMS08,Bou22}, defined here as a smooth mapping $\mathcal{R} \from \tbundle \to \points$ that agrees with the exponential map up to first order, \ie
\begin{equation}
\label{eq:retraction-def}
\tag{Rtr}
\retrbase_{\point}(0)
	= \point
	\quad
	\text{and}
	\quad
\left.\ddt\right\vert_{\ctime=0} \retrbase_{\point}(\ctime \vvec)
	= \vvec
	\quad
	\text{for all $(\point,\vvec)\in \tbundle$}.
\end{equation}
As it turns out, to replace the exponential map in \crefrange{alg:RSGM}{alg:ROG} with a retraction, we only need to slightly strengthen our assumptions on the noise in \eqref{eq:SFO}:

\begin{proposition}
\label{prop:retraction}
Suppose that the error term $\err(\point;\seed)$ of \eqref{eq:SFO} is bounded in $L^{4}$, \ie $\exof{\norm{\err(\point;\seed)}^{4}_\point} \leq \kappa^2$ for some $\kappa>0$ and all $\point\in\points$.
Then \cref{prop:algorithms} holds as stated if the exponential map in \crefrange{alg:RSGM}{alg:ROG} is replaced by a retraction.
\end{proposition}

\begin{corollary}
\label{cor:retraction}
Suppose that \cref{hyp:Hadamard,hyp:WC} hold.
If a retraction-based variant of \crefrange{alg:RSGM}{alg:ROG} is run with
a step-size schedule satisfying \eqref{eq:step}
and
an \ac{SFO} with bounded fourth moments,
then, \acl{wp1}, the generated sequence $\curr$, $\run=\running$,
converges to an \ac{ICT} set of \eqref{eq:RMD}.
\end{corollary}

\begin{remark}
Albeit relatively light, the condition $\sup_{\point}\exof{\norm{\err(\point;\seed)}^{4}_\point} < \infty$ cannot be relaxed.
This is true even in the Euclidean case with $\retrbase_{\point}(\vvec) = \point + \left(e^{\vvec}-1\right)$ and $\err(\point;\seed) \overset{\mathrm{law}}{=} \sqrt{\state} - \exof{\sqrt{\state}}$, where $\probof*{\state > \point} = \point^{-3/2}$ if $\point \geq 1$, and $\probof*{\state > \point} = 1$ otherwise.
\endenv
\end{remark}

\begin{remark}
In many practical settings, the map $\retrbase_{\point}(\vvec)$ satisfies the stronger requirement of being a ``second-order retraction'', \ie it agrees with $\exp$ up to \emph{second} order \citep{AMS08,Bou22}.
In this case, the proof technique in \cref{app:retraction} can be used to show that it suffices to have $\sup_{\point\in\points}\exof{\norm{\err(\point;\seed)}^{3}_\point}< \infty$ in \cref{prop:retraction}.
\endenv
\end{remark}

\para{Alternating variants}

The next set of variants concerns the case where $\vecfield$ is generated from a $2$-player, min-max game (the $\nPlayers$-player general-sum case is similar, but the notation is more involved so we omit it).

To state it, consider a min-max game of the form
\(
\min_{\minvar\in\minvars} \max_{\maxvar\in\maxvars}
	\minmax(\minvar,\maxvar)
\)
played over the smooth manifolds $\minvars$ and $\maxvars$ (for the min and max player respectively).
Then, instead of performing \emph{simultaneous} updates for each player in \crefrange{alg:RSGM}{alg:ROG}, a common variant is to \emph{alternate} variables according to the basic recursion
\begin{equation}
\label{eq:RARM}
\tag{RRM-alt}
\begin{alignedat}{2}
\next[\minstate]
	&= \expof{\ssstyle\curr[\minstate]}{\curr[\step] \bracks{\vecfield_{\minvar}(\curr[\minstate],\curr[\maxstate]) + \error_{\minvar,\run}}}
	\\
\next[\maxstate]
	&= \expof{\ssstyle\curr[\maxstate]}{\curr[\step]\bracks{\vecfield_{\maxvar}(\next[\minstate],\curr[\maxstate]) + \error_{\maxvar,\run}}}
\end{alignedat}
\end{equation}
where $(\vecfield_{\minvar}, \vecfield_{\maxvar}) \defeq (-\rgrad_{\minvar} \minmax, \rgrad_{\maxvar} \minmax )$,
and $\error_{\minvar,\run}$ and $\error_{\maxvar,\run}$ respectively denote the error terms entering
\eqref{eq:RRM} for the min and max player respectively.
Our next result shows that the alternating variants of \crefrange{alg:RSGM}{alg:ROG} retain the convergence properties of their simultaneous counterparts:

\begin{proposition}
\label{prop:alternation}
Suppose that \crefrange{alg:RSGM}{alg:ROG} are run with alternating updates as per \eqref{eq:RARM} and oracle and step-size parameters satisfying \eqref{eq:err} and \eqref{eq:step}.
Then \cref{prop:algorithms} holds as stated.
\end{proposition}

\begin{corollary}
\label{cor:alternation}
Suppose that \cref{hyp:Hadamard,hyp:WC} hold.
If an alternating variant of \crefrange{alg:RSGM}{alg:ROG} is run with oracle and step-size parameters satisfying \eqref{eq:err} and \eqref{eq:step},
then, \acl{wp1}, the generated sequence $\curr$, $\run=\running$,
converges to an \ac{ICT} set of \eqref{eq:RMD}.
\end{corollary}

\newmacro{\RM}{ { \mathrm{RM} } }

\begin{remark*}
A simple but useful observation is that compositions of \ac{RRM} schemes
do not change its asymptotic behavior:
specifically, for any two \ac{RRM} schemes $\RM_1$ and $\RM_2$, the update $\next = \RM_2 \circ \RM_1( \curr ) $ is equivalent to a new \ac{RRM} scheme $\curr[\tilde{\state}]$ where $\tilde{\state}_{2\run} = \state_\run$ and $\tilde{\state}_{2\run+1} = \RM_1(\tilde{\state}_{2\run})$.
This allows us to ``mix-and-match'' \crefrange{prop:algorithms}{prop:alternation} to prove, for instance, the convergence of alternating \eqref{eq:RSEG} minimizer \vs~retraction-based \eqref{eq:RPPM} maximizer in Riemannian two-player games.
\endenv
\end{remark*}

\subsection{Implications for optimization and learning}
\label{sec:implications}

We conclude this section with some concrete implications of our general theory when $\vecfield$ is specialized to specific instances that arise in optimization and learning theory.
This allows us to extend and unify several existing results, but also to obtain completely new ones altogether.

\para{Optimization on manifolds}
Perhaps the most common task for learning on manifolds is the basic minimization problem
\begin{equation}
\label{eq:opt}
\tag{Opt}
\min\nolimits_{\point\in\points}
	\obj(\point)
\end{equation}
where $\obj\from\points\to\R$ is a $C^{\vdim}$-smooth function (not necessarily geodesically convex).
In this case, applying our general theory to $\vecfield = -\nabla\obj$ yields:

\begin{proposition}
\label{prop:opt}
Suppose that \crefrange{alg:RSGM}{alg:ROG} are run against \eqref{eq:opt} with oracle and step-size parameters satisfying \eqref{eq:err} and \eqref{eq:step}.
Assume further that the problem satisfies 
\begin{enumerate*}
[\upshape(\itshape a\upshape)]
\item
\cref{asm:compact,asm:injectivity};
or, alternatively,
\item
\cref{hyp:Hadamard,hyp:WC}.
\end{enumerate*}
Then, \acl{wp1}, the induced sequence of iterates $\curr$, $\run=\running$, converges to a component of critical points of $\obj$ on which $\obj$ is constant.

If, in addition, $\sup_{\point}\exof{\norm{\err(\point;\seed)}^{4}_\point} < \infty$, the above conclusions apply to all retraction-based variants of \crefrange{alg:RSGM}{alg:ROG}.
\end{proposition}

\begin{proof}
By \cref{cor:ICT}, \cref{prop:algorithms,cor:algorithms} (or \cref{prop:retraction,cor:retraction} for the retraction-based case), it follows that $\curr$ converges to an \ac{ICT} set of \eqref{eq:RMD} \acl{wp1}.
Moreover, by Sard's theorem \citep{sternberg1999lectures}, the set of critical values of $\obj$ has empty interior, so Proposition 6.4 of \citet{Ben99} implies that every \ac{ICT} of \eqref{eq:RMD} is contained in a set of critical points of $\obj$ on which $\obj$ is constant.
Our assertion then follows by combining the above.
\end{proof}

\Cref{prop:opt} contains as a special case the analysis of \citet{Bon13} who established the almost sure convergence of \cref{alg:RSGM} (and its retraction-based variants) under the assumptions that
\begin{enumerate*}
[\upshape(\itshape a\upshape)]
\item
$\curr$ is precompact (\cref{asm:compact});
\item
the injectivity radius of $\points$ is uniformly bounded from below (\cref{asm:injectivity});
and
\item
the oracle $\orcl(\point;\seed)$ is uniformly bounded in both $\point\in\points$ and $\seed\in\seeds$, \ie $\ess\sup_{\point,\seed}\norm{\orcl(\point;\seed)}_{\point} < \infty$.
\end{enumerate*}
In this regard, \cref{prop:opt} not only relaxes the bounded oracle requirement of \citet{Bon13}, but it provides a straightforward way to establish the convergence of a wide array of Riemannian algorithms which cannot otherwise be covered by the tailor-made analysis of \citet{Bon13}.
We are not aware of a comparable convergence result for \crefrange{alg:RPPM}{alg:ROG} (or their retraction-based variants).

In addition, we note that \cref{prop:opt} includes the celebrated \acdef{NPG} method of \citet{kakade2001natural} for reinforcement learning \textendash\ which, as noted by \citet{Bon13}, is a particular case of retraction-based \cref{alg:RSGM}.
We stress that our result applies not only to ``vanilla'' \ac{NPG} methods, but also to its optimistic/extra-gradient variants.
To our knowledge, there are no comparable results in the \ac{NPG} literature.

\para{Games on manifolds}

We next move on to the game setting, \ie when $\points$ decomposes as a product of the form $\points = \points_{1}\times\dotsm\times\points_{\nPlayers}$ for some $\nPlayers\in\N$ and, likewise, the $\play$-th component of $\vecfield = (\vecfield_{1},\dotsc,\vecfield_{\nPlayers})$ is of the form $\vecfield_{\play} = \grad_{\point_{\play}} \pay_{\play}(\point_{1},\dotsc,\point_{\nPlayers})$, where $\pay_{\play} \from \points \to \R$ is the payoff function of the $\play$-th player \textendash\ for applications and a detailed discussion, \cf \citet{RBS16} and references therein. 

There are many solution concepts of games on manifolds that can be seen as a natural generalization of Nash equilibria. In this work, we will consider the so-called \acdefp{NSE} \cite{kristaly2014nash,MHC22}:

\begin{definition}
We say that $\sol \defeq (\eq_{1},\dotsc,\eq_{\nPlayers}) \in \points$ is a \emph{Nash–Stampacchia} equilibrium of $\vecfield$ if
\begin{equation}
\nn
\inner{\vecfield(\sol)}{\exp^{-1}_{\sol}(\point)}_{\sol}
	\defeq \insum_{\play=1}^{\nPlayers} \inner{\vecfield_i(\sol)}{\exp^{-1}_{\sol_{\play}}(\point_{\play})}_{\sol_{\play}}
	\geq 0,
	\quad
	\text{for all $\point \in \points$}.
\end{equation}
\end{definition}

Our first result below concerns the convergence of \crefrange{alg:RPPM}{alg:ROG} in a general class of (Riemannian) monotone games known as $\alpha$-accretive games \cite{wang2010monotone}:

\begin{proposition}
[Riemannian $\alpha$-accretive games]
\label{prop:monotone-games}
Let $\vecfield = [-\rgrad_{\point_i} \minmax_i] $ be an $\alpha$-accretive game field,
\ie all $r\geq 0$, we have
\begin{equation}
\nn
(1+\alpha r)\dist(\point,\pointalt) \leq \dist\left(\exp_{\ssstyle\point}(r \vecfield(\point)), \exp_{\ssstyle\pointalt}(r \vecfield(\pointalt))\right), \quad \alpha > 0.
\end{equation}
Then \crefrange{alg:RSGM}{alg:ROG}, as well as their alternating\,/\,retraction-based versions, converge to the game's set of \acp{NSE}.
\end{proposition}

\begin{proof}
This immediately follows by combining \crefrange{prop:algorithms}{prop:alternation} with the main result of \citep{kristaly2014nash}.
\end{proof}

To the best of our knowledge, most of the algorithms we consider are new in the setting of Riemannian monotone games except for the \emph{deterministic} gradient and extra-gradient methods \cite{tang2012korpelevich, neto2016extragradient, fan2020tseng, khammahawong2020extragradient, chen2021modified}.

While being quite general, accretivity is a strong, convexity-like assumption about the games.
In our next result, we prove general convergence for a class of \emph{non-convex} games.

\begin{proposition}
[Riemannian potential games]
\label{prop:potential-games}
Let $\vecfield = [-\rgrad_{\point_i} \minmax_i] $ be a game field associated with a Riemannian \textbf{potential} game \citep{klavins2002phase,muhammad2005decentralized}. Then \crefrange{alg:RSGM}{alg:ROG}, as well as their alternating\,/\,retraction-based versions, converge to the {critical points} of the game potential.
\end{proposition}

\begin{proof}
Simply combine \crefrange{prop:algorithms}{prop:alternation}.
\end{proof}

For Riemannian potential games, the convergence of the continuous-time dynamics \eqref{eq:RMD} is well known, but we are not otherwise aware of a similar result for stochastic, discrete-time \ac{RRM} methods.
Our theory bridges this gap by showing that the same guarantees are in fact achieved by a wide array of \ac{RRM} schemes \textendash\ not only by their continuous-time ancestor.

\newacro{MPG}{Markov potential game}

\para{Limit cycles in Riemannian manifolds}

We conclude this section by showing that, in complement to the pointwise convergence results above, our theory can also be used to derive convergence to limit cycles that arise in more general Riemannian settings.

\begin{example}[Cycling  of \ac{RRM} schemes]
\label{ex:cycling}
The following example is taken from \cite{drivas2021life}: Consider the vector field on $\mathbb{S}^2 \coloneqq \setdef{\point\in\R^3}{\point^2_1 + \point_2^2 + \point_3^3 = 1}$, defined by
\begin{equation}
\label{eq:life-death}
\vecfield(\point) = \begin{bmatrix}
-\point_2 \\ \point_1 \\ 0
\end{bmatrix} + \left(\point_3^2 - \frac{1}{4} \right)
\begin{bmatrix}
\point_1\point_3 \\ \point_2\point_3 \\ -\point_1^2-\point_2^2
\end{bmatrix}.
\end{equation}

Then it is known that the associated \eqref{eq:RMD} \emph{cycles} in the sense that the \ac{ICT} sets of $\vecfield$ in \eqref{eq:life-death} contain attracting periodic orbits; see \cref{fig:sphere}. Our \crefrange{prop:algorithms}{prop:alternation} then imply that any \ac{RRM} scheme driven by $\vecfield$ also cycles. To the best of our knowledge, this is the first rigorous example of a cycling problem for Riemannian stochastic approximation in the literature.
\end{example}


\begin{figure}[tbp]
\centering
\includegraphics[height=1.2in]{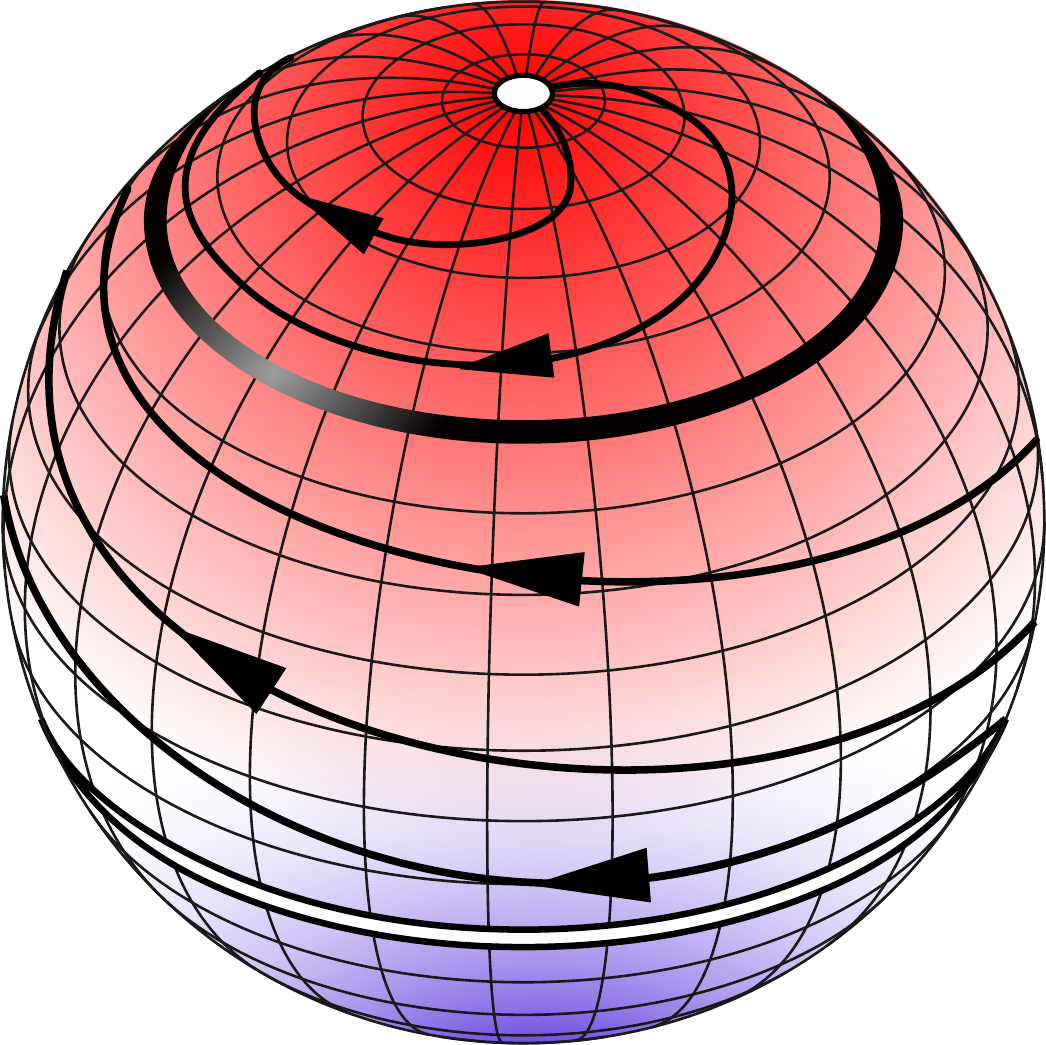}
\caption{The flow of \eqref{eq:RMD} induced by the game field \eqref{eq:life-death}.}
\label{fig:sphere}
\vspace{-2ex}
\end{figure}


\section{Concluding remarks}
\label{sec:discussion}

Our theory provides a unified analysis for the convergence of \acl{RRM} schemes that might seem vastly different from each other at first sight;
instead, by verifying certain simple criteria on the error terms $\curr[\error]$ in \cref{asm:errorbounds}, our analysis shows that we can study the deterministic dynamics \eqref{eq:RMD} to directly infer the algorithm's long-run behavior.
At the same time, our results offer but a glimpse of the flexibility of \eqref{eq:RRM}, and several important research directions remain open:
\begin{enumerate}
\item
In many applications (especially to game theory and sequential online learning), oracle access to $\vecfield$ may be out of reach, and one would need to employ \emph{zeroth-order} \textendash\ or \emph{bandit} \textendash\ optimization methods.
It this case, a key question that arises is whether a Riemannian \acli{KW} algorithm can be analyzed as a special case of \eqref{eq:RRM} \textendash\ and, if so, whether there are any fundamental differences relative to the Euclidean setting.
\item
The diminishing step-size assumption is indispensable for our analysis, which covers many practically relevant settings.
However, another common strategy involves \emph{constant} step-sizes, and this is not covered by our theory.
It would thus be interesting to see if the techniques for analyzing constant step-size \acs{SA} schemes \citep{kushner1981asymptotic,KY97} can be generalized to a manifold setting.
\item
Finally, several Riemannian algorithms are known to escape undesirable solutions \citep{criscitiello2019efficiently,sun2019escaping}.
We conjecture that the general avoidance theory of \cite{Pem90,BH95,BD96,HMC21} can be likewise extended to Riemannian manifolds;
if true, this would imply that many iterative Riemannian methods (including retraction-based ones) converge \acl{wp1} \emph{only} to local minimizers.
\end{enumerate}
We defer the study of the above questions to future work.

\section*{Acknowledgments}
\begingroup
\small
%
%
This research was supported by the SNSF grant 407540\_167212 through the NRP 75 Big Data program, and by the European Research Council (ERC) under the European Union's Horizon 2020 research and innovation programme grant aggreement No 815943. YPH acknowledges funding through an ETH Foundations of Data Science (ETH-FDS) postdoctoral fellowship. Part of this work was done while PM was visiting the Simons Institute for the Theory of Computing. PM is also grateful for financial support by the French National Research Agency (ANR) in the framework of the “Investissements d’avenir” program (ANR-15-IDEX-02), the LabEx PERSYVAL (ANR-11-LABX-0025-01), MIAI@Grenoble Alpes (ANR-19-P3IA-0003), and the grant ALIAS (ANR-19-CE48-0018-01).
\endgroup

\numberwithin{lemma}{section}		
\numberwithin{corollary}{section}		
\numberwithin{proposition}{section}		
\numberwithin{equation}{section}		
\appendix

\section{Stability analysis}
\label{app:stability}

The purpose of this appendix is to prove \cref{prop:bounded}.
To that end, let $\base\in \points$ be as in \eqref{eq:WC},
define the radial distance function $\radial(\point) \defeq \dist(\point, \base)$,
and
let $k(\point) \defeq \radial^{2}(\point)/2$. The following theorem makes it
clear under which assumptions $k$ is smooth and provides a control on its Hessian.

\begin{theorem}[{\citet[Theorem 6.6.1]{jost2017riemannian}}]
With notation as above,
suppose that the exponential map $\exp_{\base}$ is a diffeomorphism on $\setdef{\vvec\in\tspace}{\norm{\vvec}_{\point} \leq \rho}$.
Moreover, suppose that the sectional curvature of $\points$ is nonpositive and bounded below by $-\kappa^{2}$ on a geodesic ball of radius $\rho>0$ centered at $\base$. 
Then $k$ is smooth on the punctured ball $\ball_{\rho}^{\ast}(\base) \defeq \setdef{\point\neq \base}{\dist(\point,\base) \leq \rho}$ and, in particular,
\begin{equation}
\grad k(\point)
	= - \log_{\base}(\point)
\end{equation}
In addition, $\norm{\grad k(\point)} = \radial(\point)$, and
\begin{equation}
\Hess k(\point)[v,v]
	\leq \kappa \radial(\point) \ctgh(\kappa \radial(\point)) \norm{v}_{\point}^{2},
	\quad
	\text{for all $\point \in \ball_{\rho}^{\ast}(\base)$ and $\vvec \in \tspace$}.
\end{equation}
\end{theorem}

\begin{remark*}
Since $\points$ is simply connected and complete \textendash\ by \cref{hyp:Hadamard} \textendash\ we may take $\rho = \infty$ in the theorem above \citep[Corollary 6.9.1]{jost2017riemannian}.
\end{remark*}

\newcommand{\uboundfder}{C_{1}}
\newcommand{\uboundsder}{C_{2}}
\newcommand{\lyapsmooth}{C}

Now, to proceed, recall that $\vecfield$ satisfies the weak coercivity condition \eqref{eq:WC} if, for some $\coer>0$, we have
\begin{equation}
\inner{\vecfield(\point)}{\grad k(\point)}_{\point}
	\leq 0
\end{equation}
whenever $\point\in \points$ lies outside a geodesic $\ball_{\coer}(\base)$ contained in the interior of $\points$.
Our proof relies on constructing a suitable ``energy function'' that
serves as an easy-to-control proxy for the distance of the iterates of \eqref{eq:RRM} from the origin. This function will be of the form 
\begin{equation}
\lyap(\point)
	= f(\radial(\point)) 
\end{equation}
where $f$ is a $\mathcal{C}^\infty$ non-negative function with $f(x) = 0$ for all $x\leq
\coer$ and satisfies
\begin{equation}\label{eq:lyap-f}
    0 \leq f'(x) \leq \uboundfder , \quad f''(x) \leq \uboundsder
\end{equation}
for all $x \geq \coer$. Moreover, we require $f(x) = \Omega(x)$ as $x\to \infty$ so that controlling $f$ implies control of $x$ (for a concrete example of such a function, \cf \cref{lem:smooth-lyap-example}).
Our first result is that $\lyap = f\circ\radial$ has a bounded Hessian and is smooth.

\begin{lemma}
\label{lem:lyap-smoothness}
Let $\lyap$ be defined as above.
Then $\lyap$ is negatively correlated with $\vecfield$ in the sense that
\begin{equation}
\inner{\grad \lyap(\point)}{\vecfield(\point)}_{\point}
	\leq 0
	\quad
	\text{for all $\point \in \points$}.
\end{equation}
Moreover, there exists a constant $\lyapsmooth > 0$ such that $\Hess \lyap(\point)[v,v] \leq \lyapsmooth\norm{v}_{\point}^{2}$, and hence
\begin{equation}
\label{eq:lyap-smoothness}
\lyap(\pointalt)
	\leq \lyap(\point)
		+ \inner{\grad \lyap(\point)}{\log_{\point}(\pointalt)} 
		+ \frac{C}{2} \dist^{2}(\point, \pointalt)
	\quad
	\text{for all $\point,\pointalt \in \points$}.
\end{equation}
\end{lemma}

\begin{proof}
We begin by noting that the gradient of $\lyap$ is given by
\begin{equation}
\grad\lyap(\point)
	= \begin{cases}
    	0
			&\quad
			\text{if $\radial(\point) \leq \coer$},
		\\
        \frac{f'(\radial(\point))}{\radial(\point)}\grad k(\point)
        	&\quad
			\text{if $\radial(\point) > \coer$}.
    \end{cases}
\end{equation}
By assumption, $f'(\radial(\point))/\radial(\point) \geq 0$ so $\inner{\grad \lyap(\point)}{\vecfield(\point)}_{\point}$ has the same sign as $\inner{\grad k(\point)}{\vecfield(\point)}_{\point}$ if $\radial(\point) > \coer$ and is otherwise zero if $\radial(\point) \leq \coer$.
We thus conclude that $\lyap$ and $\vecfield$ are negatively correlated, as claimed.

Now, to compute the Hessian of $\lyap$, notice that $\Hess \lyap(\point)[v,v] = \inner{\grad_v
\grad \lyap(\point)}{v}_{\point}$. Hence,
\begin{align}
\Hess \lyap(\point)[v,v]
	&=\grad_v \frac{f'(\radial(\point))}{\radial(\point)} \cdot \inner{\grad k(\point)}{v}_{\point} 
		+ \frac{f'(\radial(\point))}{\radial(\point)} \inner{\grad_v \grad k(\point)}{v}_{\point}
	\notag\\
	&= \underbrace{\inner*{\grad \frac{f'(\radial(\point))}{\radial(\point)}}{v} \cdot \inner{\grad k(\point)}{v}_{\point}}_{\termi}
		+ \underbrace{\frac{f'(\radial(\point))}{\radial(\point)} \, \Hess k(\point)[v,v]}_{\termii}.
\end{align}
Here we use the same notation for directional derivative of a scalar function
and the covariant derivative.
With this in mind, the first step in computing $\termi$ is the observation that
\begin{equation}
\grad \frac{f'(\radial(\point))}{\radial(\point)}
	= \left(f''(\radial(\point)) - \frac{f'(\radial(\point))}{\radial(\point)}\right)\frac{1}{\radial^{2}(\point)} \grad k(\point),
\end{equation}
and hence
\begin{equation}
\termi
	= \left(f''(\radial(\point)) - \frac{f'(\radial(\point))}{\radial(\point)}\right)\frac{1}{\radial^{2}(\point)}\inner{\grad k(\point)}{v}_{\point}^{2}
	\leq \frac{\uboundsder}{\radial^{2}(\point)} \norm{\grad k(\point)}_{\point}^{2} \norm{v}_{\point}^{2}
	= \uboundsder\, \norm{v}_{\point}^{2}.
\end{equation}
For $\termii$, as $x\ctgh x \leq 1 + x$ for $x\geq 0$, we obtain
\begin{equation}
\termii
	\leq \frac{f'(\radial(\point))}{\radial(\point)} (1 + \kappa \radial(\point)) \norm{v}_{\point}^{2}
	\leq \uboundfder \left(1/\coer + \kappa\right) \norm{v}_{\point}^{2}.
\end{equation}
Summing up everything, we obtain
\begin{equation}
\label{eq:Hess-upper}
\Hess \lyap(\point)[v,v]
	\leq (\uboundsder + \uboundfder/\coer + \uboundfder \kappa)\norm{v}_{\point}^{2}
	\eqdef \lyapsmooth\norm{v}_{\point}^{2},
\end{equation}
that is, $\lyap$ has bounded Hessian. Moreover, $\lyap$ is smooth as a composition of smooth
functions. Let $\point, \pointalt \in \points$ be arbitrary, and let
$\gamma:[0,1]\to \points$ be a geodesic connecting the two. By Taylor's remainder
theorem, there exists some $t\in(0,1)$ such that
\begin{equation}
\lyap(\pointalt)
	= \lyap(\point)
		+ \inner{\grad \lyap(\point)}{\dot{\gamma}(0)}_{\point}
		+ \frac{1}{2} \Hess \lyap(\gamma(t))[\dot{\gamma}, \dot{\gamma}].
\end{equation}
Thus, invoking \eqref{eq:Hess-upper} and noting that $\norm{\dot{\gamma}} = \dist(\point, \pointalt)$ and $\dot{\gamma}(0) = \log_{\point}(\pointalt)$, we obtain \eqref{eq:lyap-smoothness} and our proof is complete.
\end{proof}

We now proceed to the main argument, where we show how to use $\lyap$ to control the
iterates $\curr$. Letting $\curr[\lyap] = \lyap(\curr)$ and using
\cref{lem:lyap-smoothness} we get
\begin{align}
\label{eq:Lyap-bound1}
\next[\lyap] &= \lyap\left(\exp_{\curr}(\curr[\step] \curr[\signal])\right) \notag \\
	&\leq \curr[\lyap]
		+ \curr[\step] \inner{\grad \lyap(\curr)}{\curr[\signal]}_{\curr}
		+ \frac{\lyapsmooth \curr[\step]^{2}}{2} \norm{\curr[\signal]}_{\curr}^{2}
	\notag\\
	&\leq \curr[\lyap]
		+ \curr[\step] \inner{\grad \lyap(\curr)}{\curr[\noise] + \curr[\bias]}_{\curr}
	+ \frac{3 C \curr[\step]^{2}}{2}
    \bracks[\big]{\twonorm{ \vecfield(\curr)}_{\curr}^{2} + \twonorm{\curr[\noise]}_{\curr}^{2} +
\twonorm{\curr[\bias]}_{\curr}^{2}}
\end{align}
where the second line follows from the negative correlation of $\lyap$ and $\vecfield$, the
definition \eqref{eq:signal} of $\curr[\signal]$, and the Cauchy-Schwarz inequality.
Conditioning on $\curr[\filter]$ and taking expectations, and invoking Cauchy-Schwarz and the fact that
$\norm{\grad \lyap(\curr)} \leq \frac{\uboundfder}{\radial(\curr)} \norm{\grad k(\curr)} = \uboundfder$, we obtain:
\begin{equation}
\label{eq:Lyap-bound2}
\exof{\next[\lyap] \given \curr[\filter]}
	\leq \curr[\lyap] +  \curr[\step]\uboundfder \curr[\bbound]
		+ \tfrac{3}{2}\lyapsmooth \curr[\step]^{2} \bracks[\big]{\vbound^{2}
		+ \curr[\bbound]^{2}
		+ \curr[\sdev]^{2}},
\end{equation}
where we have bounded the second moments by their respective upper bounds.

To proceed, let $\curr[\eps] = \curr[\step]\uboundfder \curr[\bbound] + (3/2)\lyapsmooth \curr[\step]^{2} \bracks[\big]{\vbound^{2} + \curr[\bbound]^{2} + \curr[\sdev]^{2}}$ denote the ``residual'' term in \eqref{eq:Lyap-bound2}.
Notice that
\begin{equation}
\sum_{\run=\start}^{\infty} \curr[\eps]
	\leq \uboundfder\sum_{\run=\start}^{\infty} \curr[\step]\curr[\bbound]
		+ \frac{3\lyapsmooth}{2} \sum_{\run=\start}^{\infty} \curr[\step]^{2} (\vbound^{2} + \curr[\bbound]^{2} + \curr[\sdev]^{2}),
\end{equation}
and hence, by \cref{asm:errorbounds} and the dominated convergence theorem, we infer that  $\exof{\sum_{\run} \curr[\eps]} < \infty$.

Next, consider the auxiliary process $\curr[\lyapalt] = \curr[\lyap] + \exof{\sum_{\runalt=\run}^{\infty} \iter[\eps] \given \curr[\filter]}$, adapted to the same filtration.
By \eqref{eq:Lyap-bound2}, we have $\exof{\next[\lyapalt] \given \curr[\filter]}
\leq \curr[\lyap] + \exof{\sum_{\runalt=\run}^{\infty} \iter[\eps] \given \curr[\filter]}= \prev[\lyapalt]$, \ie
$\curr[\lyapalt]$ is a supermartingale with respect to $\curr[\filter]$.
This shows that $\exof{\curr[\lyapalt]} \leq \exof{\init[\lyapalt]} < \infty$, \ie $\curr[\lyapalt]$ is uniformly bounded in $L^{1}$.
Hence, by Doob's supermartingale convergence theorem \citep[Theorem~2.5]{HH80}, it follows that $\curr[\lyapalt]$ converges \acl{wp1} to some finite random limit $\lyapalt_{\infty}$.
In turn, since $\sum_{\run} \curr[\eps] < \infty$ \as, this implies that
$\curr[\lyap] = \curr[\lyapalt] - \exof{\sum_{\runalt=\run}^{\infty} \curr[\eps] \given \curr[\filter]}$ also converges to some (random) finite limit \as.
From this and the fact that $\curr[\lyap] = \Omega(\radial(\curr))$, we deduce $\limsup_{\run} \radial(\curr) < \infty$ as claimed.

\begin{lemma}
\label{lem:smooth-lyap-example}
Let $h\from \R \to \R$ be the function
\begin{equation}
h(x)
	= \begin{cases}
    	0
			&\quad
			\text{if $x \leq 0$}
		\\
		\frac{e^{-1/x}}{e^{-1/x}+e^{-1/(1-x)}}
			&\quad
			\text{if $x\in(0,1)$}
		\\
		1
			&\quad
			\text{otherwise}
		\end{cases}
\end{equation}
and, for $\coer>0$, let
\(
f(x)
	= \int_{0}^{x} h(s-\coer) \dd s.
\)
Then $f$ is $C^\infty$ and it satisfies the conditions \eqref{eq:lyap-f} with $\uboundfder = 1$ and $\uboundsder = 2$.
In addition, one has $f(x) \geq x - (R+1)$, and hence $f(x) = \Omega(x)$.
\end{lemma}

\begin{proof}
As $h(x) \in [0,1]$, we obtain that $f'(x) \in [0,1]$. 
By a straightforward computation, one observes that the first derivative of $h$ is bounded as $0 \leq h'(x) \leq 2$, so
\(
f''(x) = h'(x-\coer) \leq 2.
\)
Then, to complete our proof, simply notice that, for $x \geq R+1$, we have $f(x) = \int_0^x h(s-R) \dd s \geq \int_{R+1}^x 1 \dd s = x-(R+1)$, as claimed.
\end{proof}

\section{Geometric preliminaries for the proof of \cref{thm:APT}}
\label{app:prelim}

In this appendix, we collect the necessary technical prerequisites for the proof of \cref{thm:APT}.

\subsection{The \acl{PFS}}
\label{app:pcoord}

We begin with the definition of the frame system that we use to compare vectors on different tangent spaces.
To that end, fix any two points $\point, \pointalt \in \points$, and consider two arbitrary vectors $\vvec \in \tspace$ and $\vvecalt \in \tspace[\pointalt]$.
There is a convenient frame system (\ie a set of bases for $\tspace$ and $\tspace[\pointalt]$) for comparing $\vvec$ and $\vvecalt$, defined as follows: Pick an arbitrary orthonormal frame $\braces*{\bvec_\coord}_{\coord=1}^\vdim$ for $\tspace$.
Since the parallel transport map is an isometry, the vectors $\braces{\bvec'_\coord}_{\coord=1}^\vdim \defeq \braces*{ \ptrans{\point}{\pointalt}{\bvec_\coord} }_{\coord=1}^\vdim$ form an orthonormal basis for $\tspace[\pointalt]$.
Now consider the components of $\vvec, \vvecalt$ in these two frames, namely
\begin{align}
\label{eq:parallel}
	\vvec_{\coord}^{\pcoord}
		\defeq \inner{\vvec}{ \bvec_\coord}_\point
	\;
	\text{and}
	\;
	\vvecalt_{\coord}^{\pcoord}
		\defeq \inner{\vvecalt}{ \bvec'_\coord}_{\pointalt}
	\quad
	\text{for all $\coord=1,\dotsc,\vdim$}.
\end{align}
We shall call \eqref{eq:parallel} the \emph{\acl{PFS}} for vectors $\vvec$ and $\vvecalt$ (the dependence on the initial frame $\braces*{\bvec_\coord}_{\coord=1}^\vdim$ is suppressed).

By virtue of parallel transport, in the \acl{PFS} we have
\begin{align}
\label{eq:comp_in_pcoord}
\ptrans{\point}{\pointalt}{\vvec}_{\coord}^{\pcoord}
	=  \inner{\ptrans{\point}{\pointalt}{\vvec}}{\bvec'_\coord}_{\pointalt}
	= \inner{\vvec}{ \ptrans{\pointalt}{\point}{\bvec'_\coord}}_{\point}
	= \vvec_{\coord}^{\pcoord}.
\end{align}
In the same vein, we have $\ptrans{\pointalt}{\point}{\vvecalt}_{\coord}^{\pcoord} = \vvecalt_{\coord}^{\pcoord}$.
Thus, since $\braces*{\bvec_\coord}_{\coord=1}^\vdim$ is orthonormal, we may write:
\begin{equation}
\norm{\vvecalt - \ptrans{\point}{\pointalt}{\vvec} }_\point
	= \norm{  \vvec - \ptrans{\pointalt}{\point}{\vvecalt} }_{\pointalt}
	= \norm{ \vvec^{\pcoord} - \vvecalt^{\pcoord}}_{2}.
\end{equation}
In other words, in the \acl{PFS}, the comparison of vectors living on different tangent spaces is reduced to simply comparing the Euclidean norms of their components.
For instance, the $\lips$-Lipschitzness of $\vecfield$ can be rephrased as
\begin{equation} 
\label{eq:vec_in_pcoord}
\norm{\vecfield^{\pcoord}(\pointalt) - \vecfield^{\pcoord}(\point) }_{2}
	\leq \lips \dist(\point, \pointalt)
	\quad
	\text{for all $\point,\pointalt\in\points$}.
\end{equation}

\subsection{The \acl{FCS}}
\label{app:fcoord}

For any $h$, let $\nbhd_h \subset \tspace[\apt{\ctime+h}] \simeq \R^d$ be a neighborhood of 0 on which the mapping 
\begin{equation}
\label{eq:normal_nbhd}
\exp_{\apt{\ctime+h}} \from \nbhd_h \to \points
\end{equation}
is a diffeomorphism between $\nbhd_h$ and $\exp_{\apt{\ctime+h}} \left( \nbhd_h \right)$.
It is well-known that such a neighborhood exists, and that the exponential map $\exp_{\apt{\ctime+h}}$ along with an arbitrary orthonormal frame at $\tspace[\apt{\ctime+h}]$ induces a local coordinate system on $\exp_{\apt{\ctime+h}} \left( \nbhd_h \right)$, called the \emph{normal coordinate frame with center $\apt{\ctime+h}$} \citep{Lee97}.
Normal coordinates are best-suited for comparing \emph{distances} of points on manifolds:
for instance, if $\fpoint'$ is the normal coordinate of $\pointalt$ with center $\point$, then $\dist(\point,\pointalt) = \norm{\fpoint'}_{2}$.

The \acdef{FCS} \citep{manasse1963fermi}, roughly speaking, is a system of normal coordinates ``along a curve''.
To define it, fix an arbitrary orthonormal frame $\braces*{\bvec_\coord(0)}_{\coord=1}^\vdim$ for $\tspace[\apt{\ctime}]$.
We can obtain a system of orthonormal frames $\braces*{\bvec_\coord(h)}_{\coord=1}^\vdim$ by parallel transporting $\braces*{\bvec_\coord(0)}_{\coord=1}^\vdim$ from $\tspace[\apt{\ctime}]$ to $\tspace[\apt{\ctime+h}]$ along the curve $h \mapsto\apt{\ctime + h}$.
Let $\nbhd_h \subset \tspace[\apt{\ctime+h}]$ be a neighborhood of 0 defined as in \eqref{eq:normal_nbhd}, and set $\fnbhd \defeq \union_h \braces{\exp_{\apt{\ctime+h}} \left( \nbhd_h \right)} \subset \points$.
Finally, consider the mapping 
\begin{equation}
\label{eq:fdiffeo}
\fdiffeo\from\R_{+} \times \fnbhd \to \R_{+} \times \R^\vdim
\end{equation}
that sends a point $(h,\point) \in \R_{+} \times \fnbhd$ to $(h,\fpoint) \in \R_{+} \times \R^d$, where $\fpoint$ is the normal coordinate of $\point$ with center $\apt{\ctime+h}$ and frame $\braces*{\bvec_\coord(h)}_{\coord=1}^\vdim$.
By virtue of the normal coordinates, we know that $\fdiffeo$ is a diffeomorphism between $\R_{+} \times \fnbhd$ and a neighborhood of $\R_{+} \times \{0\}$.
The mapping $\fdiffeo$ and its inverse is called the \acli{FCS} along the curve $h \mapsto \apt{\ctime+h}$.
In the sequel, we will abuse the notation and simply call it the Fermi coordinates along $\state$.

The following property of the \acl{FCS} plays a key role in our analysis and we will use it freely:
\begin{lemma}[\citefull{takahashi1981probability}; \citefull{fujita1982onsager}]
\label[lemma]{lem:fcoord}
Let $\curve$ be a differentiable curve on $\points$ such that $\curve(h) \in \nbhd_h$ for all $h \in \R_{+}$, and let $\fcurve$ be the curve of $\curve$ in the \ac{FCS} along $\state$ (\ie $(h,\fcurve(h)) = \fdiffeo(h, \curve(h))$.
Then
\begin{equation}
\dot{\fcurve}_{\coord}(h)
	= \dot{\curve}_{\coord} (h) - \dot{\state}_{\coord} (\ctime+h)
		+ \bigoh\parens[\big]{\norm{\fcurve(h)}_{2}^{2}}
\end{equation}
where $\dot{\state}_{\coord}(\ctime+h) \defeq \inner{\dot{\state}(\ctime+h)}{ \bvec_\coord(h)}_{\apt{\ctime+h}}$ is the $\coord$-th component of $\dot{\state}(\ctime+h)$ in the frame $\braces*{\bvec_\coord(h)}_{\coord=1}^\vdim$,
and
$\dot{\curve}_{\coord}(h)$ is the $\coord$-th component of $\dot{\curve}(h)$ in the \textpar{possibly non-orthonormal} frame induced by the normal coordinate system with center $\apt{\ctime+h}$ and frame $\braces*{\bvec_\coord(h)}_{\coord=1}^\vdim$.
\end{lemma}

\subsection{Comparing the differential of $\exp$ and parallel transport}
\label{app:compare}

As will become clear in the proof, the \ac{PFS} is convenient for comparing \emph{vectors} at different points, whereas the \ac{FCS} is best suitable for comparing the \emph{distance} between curves, both features being essential to our proof.
There is, however, a dichotomy: It is known that if the Fermi coordinate in \eqref{eq:fdiffeo} is simultaneously a \ac{PFS} for all points nearby the curve $\state(\cdot)$, then the underlying manifold $\points$ must be flat; \ie the Riemannian curvature tensor vanishes everywhere \citep{iliev2006handbook}.

Therefore, we would like to work with parallel frame and Fermi coordinate systems separately, and compare the difference between the two whenever needed.
To this end, we will need the following technical lemma, whose proof can be found in \citep[Theorem 3.12]{lezcano2020curvature} or \citep[Proposition A.1]{criscitiello2020accelerated}:

\begin{lemma}[Comparing $\drm\exp$ and parallel transport]
\label[lemma]{lem:dexp_vs_pt}
Let $\points$ be a Riemannian manifold whose sectional curvatures are in the interval $[K_\textup{low}, K_\textup{up}]$, and let $K = \max\left(\lvert K_\textup{low}\rvert, \lvert K_\textup{up}\rvert \right)$.
For $\vvec \in \tspace$, consider the geodesic $\curve(t) = \exp_{\point}(t\vvec)$.
If $\curve$ is defined and has no interior conjugate point on the interval $[0,1]$, then
\begin{equation}
\label{eq:dexp_vs_pt}
\forall \vvecalt \in \tspace, \quad\quad \norm{T_\vvec(\vvecalt) - \Gamma_\vvec(\vvecalt)}_{\curve(1)} \leq K \cdot f_{K_\textup{low}}( \norm{\vvec}_\point )\cdot\norm{\vvecalt_\perp}_{\point}
\end{equation}where $\vvecalt_\perp \defeq \vvecalt - \frac{\inner{\vvec}{\vvecalt}_\point}{\inner{\vvec}{\vvec}_\point}\vvec$ is the component of $\vvecalt$ orthogonal to $\vvec$, $T_\vvec = \drm\exp_\point(\vvec)$ is the differential of the exponential map, and $\Gamma_\vvec$ denotes the parallel transport along $\curve$ from $\curve(0)$ to $\curve(1)$.
The function $f_{K_\textup{low}}$ in \eqref{eq:dexp_vs_pt} is defined as
\begin{equation}
\label{eq:fdef}
f_{K_\textup{low}}(a)
	= \begin{cases}
		\frac{r^2}{6}
			&\quad
			\text{if $K_\textup{low} = 0$},
		\\
		r^2\left( 1- \frac{\sin({a}/{r})}{{a}/{r}} \right)
			&\quad
			\text{if $K_\textup{low} = \frac{1}{r^2} > 0$},
		\\
		r^2\left( \frac{\sinh({a}/{r})}{{a}/{r}}-1 \right)
			&\quad
			\text{if $K_\textup{low} = - \frac{1}{r^2} < 0$}.
	\end{cases}
\end{equation}
Moreover, the function $f_{K_\textup{low}}$ is dominated by the case $K_\textup{low} < 0$: For all $K \geq \abs{K_\textup{low}}$ and $a \in \R_{+}$, 
\begin{equation}
\label{eq:fKbound}
f_{K_\textup{low}}(a) \leq f_{-K}(a).
\end{equation}
\end{lemma}

%

\section{Proofs for \cref{sec:applications}}
\label{app:applications}


\subsection{Proof of \cref{prop:algorithms}}
\label{app:algorithms}


We prove \cref{prop:algorithms} in this appendix. To this end, we first provide a convenient lemma which shows that, almost surely, the effect of the noise is asymptotically annihilated by the step-size:
\begin{lemma}
\label[lemma]{lem:vanishing-noise}
Under the assumptions in \cref{prop:algorithms}, we have, with probability 1,
\begin{equation}
\label{eq:vanishing-noise}
\lim_{\run \to \infty} \norm*{\curr[\step]  \orcl(\curr; \curr[\seed])   }_{\ssstyle \curr}  = 0.
\end{equation}
\end{lemma}

\begin{proof}

By definition,
\begin{align}
\nn
\norm*{\curr[\step]  \orcl(\curr; \curr[\seed])   }_{\ssstyle \curr} &= \norm*{\curr[\step]  \vecfield(\curr) + \err(\curr;\curr[\seed])  }_{\ssstyle \curr} \\
\label{eq:hh}
&\leq \curr[\step] \vbound + \curr[\step] \norm*{\err(\curr;\curr[\seed])}_{\ssstyle \curr}.
\end{align}The first term goes to 0 by choice of step-sizes. To control the second term, note that by Chebyshev's inequality and \eqref{eq:err}, we have
\begin{equation}
\label{eq:noise-threshold}
\prob\left( {\norm{\err(\curr;\curr[\seed])}_{\ssstyle \curr} \geq \noiselevel}\right)
	\leq \frac{\sdev^2}{\run \log^{1+\frac{\epsilon}{2}} \run}
\end{equation}
where $\eps$ is the same as in our choice of step-size in \cref{prop:algorithms}.
In turn, this implies that 
\[
\sum_{\run = 2}^\infty\prob\left( {\norm{\err(\curr;\curr[\seed])}_{\ssstyle \curr} \geq \noiselevel}\right) < \infty
\]
so, by the Borel-Cantelli lemma, we have $\norm{\err(\curr;\curr[\seed])}_{\ssstyle \curr} = \bigoh\left( \noiselevel \right)$ with probability $1$.
Hence, by our assumptions for the method's step-size, we get
\begin{equation}
\nn
\curr[\step] \norm{\err(\curr;\curr[\seed])}_{\ssstyle \curr}
	= \bigoh\parens*{\frac{\noiselevel}{ \sqrt{\run \log^{1+\eps} \run} }}
	= \bigoh\parens*{\frac{1}{    \log^{\frac{\eps}{4}}\run  } }
\end{equation}
which, combined with \eqref{eq:hh}, implies \eqref{eq:vanishing-noise}.
\end{proof}
We are now ready for the full proof. 




\medskip
The second claim of \cref{prop:algorithms} is a direct consequence of \cref{lem:vanishing-noise} and \cref{asm:injectivity}.
As for the first, note that $\sum_n \curr[\step]^2  < \infty$ by our choice of step-sizes so \cref{asm:step} holds by construction and we are left to establish the summability conditions \eqref{eq:errorbounds} of \cref{asm:errorbounds}. To this end, by \cref{asm:step}, \eqref{eq:err}, and \cref{lem:vanishing-noise}, it suffices to show that $\norm{\curr[\bias]}_{\ssstyle\curr} = \bigoh(\curr[\step]\cdot \text{noise})$ where ``noise'' is a query from \eqref{eq:SFO} at the appropriate state.
We proceed method-by-method:

\para{\cref{alg:RSGM}: \Acl{RSGM}}

For \eqref{eq:RSGM}, we have $\curr[\error] = \curr[\noise] = \err(\curr;\curr[\seed])$ and $\curr[\bias] = 0$, so \eqref{eq:errorbounds} follows from the stated assumptions for \eqref{eq:SFO}.

\para{\cref{alg:RPPM}: \Acl{RPPM}}

For \eqref{eq:RPPM}, we have $\curr[\noise] = 0$ and 
\begin{align}
\norm{\curr[\bias]}_{\ssstyle\curr}
	&= \norm{  \ptrans{\next}{\curr}{\vecfield(\next)} - \vecfield(\curr) }_{\ssstyle\curr}
	\notag\\
	&\leq \lips  \dist \parens*{ \curr , \next  }
	= \curr[\step] \lips \norm{  \vecfield(\curr) }_{\ssstyle \curr}
	\leq \curr[\step] \lips M
	= \bigoh ( \curr[\step])
\end{align}
where we have used the $\lips$-Lipschitzness and $\vbound$-boundedness of $\vecfield$, and the distance-minimizing property of $\exp$ within the injectivity radius.

\para{\cref{alg:RSEG}: \Acl{RSEG}}

For the \ac{RSEG} algorithm, the recurrence \eqref{eq:RSEG} gives $\curr[\noise] = \ptrans{\ssstyle\lead}{\ssstyle\curr}{ \err(\lead;\lead[\seed])}$ so 
\begin{equation}
\exof[\big]{ \norm*{\curr[\noise]}^2_{\ssstyle\curr}}
	= \exof[\big]{\norm*{ \err(\lead;\lead[\seed]) }^2_{\ssstyle\lead}}
	\leq \sdev^{2}
\end{equation}
by \eqref{eq:err} and the fact that the parallel transport map is a linear isometry. 
Finally, for the second part of \eqref{eq:errorbounds}, the definition of \eqref{eq:RSEG} yields
\begin{equation}
\norm{\curr[\bias]}
	= \norm{ \ptrans{\ssstyle\lead}{\ssstyle\curr}{   \vecfield(\lead)}  - \vecfield(\curr)}_{\ssstyle\curr}
	\leq \lips \dist\parens*{\lead, \curr}
	= \curr[\step] \lips\norm{\orcl(\curr;\curr[\seed])}_{\ssstyle \curr}
\end{equation}
so our claim follows from the assumptions on $\orcl$.

\para{\cref{alg:ROG}: \Acl{ROG}}

For the \acs{ROG} algorithm, the recurrence \eqref{eq:ROG} gives
$\curr[\noise] = \ptrans{\ssstyle\lead}{\ssstyle\curr}{\err(\curr;\lead[\seed])}$
and
$\curr[\bias] = \ptrans{\ssstyle\lead}{\ssstyle\curr}{   \vecfield(\lead)}  - \vecfield(\curr)$,
so \cref{asm:errorbounds} can be checked exactly as in the case of \cref{alg:RSEG} above.


\subsection{Proof of \cref{prop:retraction}}
\label{app:retraction}

By definition, $\retrbase_{\point}(\vvec)$ is a smooth map and hence satisfies $\lim_{\vvec\to 0} \retrbase_{\point}(\vvec) = \point $. Then \cref{lem:vanishing-noise} readily implies that $\next$ lies in the injectivity radius of $\curr$ with probability 1 for $\run$ large enough.



We first consider the retraction-based \cref{alg:RSGM}:
\begin{equation}
\label{eq:retraction-sgd}
\begin{alignedat}{2}
\next
	&= \retrbase_{\curr}(\curr[\step] \orcl(\curr;\curr[\seed])).
\end{alignedat}
\end{equation}
Let $\curr[\tilde{\vecfield}] \in\tspace[\curr]$ be the vector such that $\exp_{\ssstyle \curr}\parens*{  \curr[\step]  \curr[\tilde{\vecfield}] } = \next$, \ie 
\begin{equation}
\label{eq:retrinverse}
\curr[\step]  \curr[\tilde{\vecfield}] = \exp^{-1}_{\ssstyle \curr} \Big( \retrbase_{\curr}(\curr[\step] \orcl(\curr;\curr[\seed]))\Big).
\end{equation}
Then \eqref{eq:retraction-sgd} is an \ac{RRM} scheme with $\curr[\error] = \curr[\tilde{\vecfield}] - \vecfield(\curr)$ where $\curr[\tilde{\vecfield}]$ is defined in \eqref{eq:retrinverse}. We will show that, under the assumption $\exof{\norm{\err(\point;\seed)}^{4}_\point}< \infty$, the following holds with probability 1:
\begin{align}
\label{eq:rtr-error-bound}
\curr[\bias] = \exof*{\curr[\error] \given \curr[\filter] }\to 0, \quad \sup_{\run}\exof*{\norm{\curr[\error]}^2_{\ssstyle\curr} } < \infty
\end{align}which obviously implies \eqref{eq:errorbounds}.

\renewcommand{\curve}{c}

Consider the curve $\curve(\ctime) \defeq \retrbase_{\curr}(\ctime \orcl(\curr;\curr[\seed]))$. By \cref{lem:vanishing-noise}, for $\run$ large enough, $\curve(\ctime)$ lies in the injectivity radius of $\curr$ almost surely for all $ \ctime \in [0,\curr[\step]] $. Let $\hat{\curve}(\ctime)$ be the smooth curve of $\curve(\ctime)$ in the normal coordinate with base $\curr$ and an arbitrary orthonormal frame, and let $\hat{\state}_{\run+1}$ be the normal coordinate of $\next$. Also, let $\curr[\tilde{\vecfield}]^{\mathrm{N}}$ be the (Euclidean) vector of $\curr[\tilde{\vecfield}]$ expanded in the chosen orthonormal basis, and define $\orcl^{\mathrm{N}}(\curr;\curr[\seed])$ and $\err^{\mathrm{N}}(\curr;\curr[\seed])$ similarly. By definition, $\hat{\state}_{\run+1}$ is nothing but $\curr[\step] \curr[\tilde{\vecfield}]^{\mathrm{N}} $.

Since $\curr = \curve(0)$ and $\next = \curve(\curr[\step])$, by properties of a retraction map we must have
\begin{align}
\nn
\curr[\step] \curr[\tilde{\vecfield}]^{\mathrm{N}} &= \hat{\curve}(\curr[\step]) \\
\nn
&=\hat{\curve}(0) +  \curr[\step]\dot{\hat{\curve}}(0)  + \bigoh\parens*{ \curr[\step]^2  \norm{\dot{\hat{\curve}}(0)}_2^2 } \\
\nn
&= \curr[\step] \orcl^{\mathrm{N}}(\curr;\curr[\seed]) +  \bigoh\parens*{ \curr[\step]^2  \norm{ \orcl(\curr;\curr[\seed]) }_{\ssstyle\curr}^2 } \\
&\eqdef \curr[\step] \orcl^{\mathrm{N}}(\curr;\curr[\seed]) + \curr[\step] \tilde{\bias}_{\run}
\end{align}where $\tilde{\bias}_{\run} = \bigoh\parens*{ \curr[\step]  \norm{ \orcl(\curr;\curr[\seed]) }_{\ssstyle\curr}^2 }$.
Therefore, since $\exof{\norm{\err(\point;\seed)}^{4}_\point}< \infty$ for all $\point \in \points$, we have
\begin{equation}
\exof*{  \norm*{ \curr[\error] }_{\ssstyle\curr}^2 }
	=  \exof*{   \norm*{ \err^{\mathrm{N}}(\curr;\curr[\seed]) + \tilde{\bias}_{\run}   }_2^2 }
	< \infty.
\end{equation}
On the other hand,
\begin{equation}
\norm*{\curr[\bias]}_{\ssstyle\curr}
	= \norm*{\exof*{\curr[\error] \given \curr[\filter] }}_{\ssstyle\curr}
	=  \norm*{\exof*{  \tilde{\bias}_{\run}   \given \curr[\filter] }}_2
	= \bigoh\parens*{ \curr[\step]  \norm{ \orcl(\curr;\curr[\seed]) }_{\ssstyle\curr}^2 }.
\end{equation}
By Chebyshev's inequality and the fact that $\exof{\norm{\err(\point;\seed)}^{4}_\point} \leq \kappa^2 < \infty$, we have
\begin{equation}
\nn
\prob\left( {\norm{\err(\curr;\curr[\seed])}^2_{\ssstyle \curr} \geq \noiselevel}\right)
	\leq \frac{\kappa^2}{\run \log^{1+\frac{\epsilon}{2}} \run}
\end{equation}
where $\eps$ is the same as in our choice of step-size in \cref{prop:algorithms}. Using an calculation identical to \cref{lem:vanishing-noise}, we conclude that 
\begin{equation}
\nn
\curr[\step] \norm{\err(\curr;\curr[\seed])}^2_{\ssstyle \curr}
	= \bigoh\parens*{\frac{\noiselevel}{ \sqrt{\run \log^{1+\eps} \run} }}
	= \bigoh\parens*{\frac{1}{    \log^{\frac{\eps}{4}}\run  } }
\end{equation}
which concludes the proof of \eqref{eq:rtr-error-bound}. 

For the retraction-based variants of \crefrange{alg:RPPM}{alg:ROG}, by the above analysis, we may replace $\retrbase_{\point}(\curr[\step] \orcl(\cdot;\curr[\seed]))$ with $\exp_{\ssstyle\point}\parens*{   \curr[\step] \parens*{ \orcl(\cdot;\curr[\seed]) + \curr[\tilde{\bias}] }}$ where $\curr[\tilde{\bias}] \to 0$ almost surely. The rest is the same as in the proof of \cref{prop:algorithms}.\hfill $\blacksquare$

\subsection{Proof of \cref{prop:alternation}}
\label{app:alternation}
Similar to \cref{prop:algorithms}, \cref{lem:vanishing-noise} guarantees that all geodesics are minimizing and invertible. Hence, by the $\lips$-Lipschitzness of $\vecfield$ and \eqref{eq:RARM}, we have, with probability 1,
\begin{align}
\nn
\norm*{\vecfield_{\maxvar}(\next[\minstate],\curr[\maxstate])-\vecfield_{\maxvar}(\curr[\minstate],\curr[\maxstate]) }_{ \curr[\maxvar] }
&\leq \lips\dist\parens*{   \begin{bmatrix}
\next[\minstate] \\ 
\curr[\maxstate]
\end{bmatrix},   \begin{bmatrix}
\curr[\minstate] \\
\curr[\maxstate]
\end{bmatrix}} \\
\nn
&= \lips \curr[\step] \norm*{\vecfield_{\minvar}( {\curr[\minstate]},\curr[\maxstate]) + \noise_{\minvar,\run} + \bias_{\minvar,\run}}_{\curr[\minstate]} \\
\label{eq:hhhh}
&\to 0
\end{align}
by \cref{lem:vanishing-noise} and \cref{prop:algorithms}. Therefore, we may rewrite \eqref{eq:RARM} as
\begin{align}
\next = \exp_{\ssstyle\curr} \parens*{  \curr[\step] \begin{bmatrix}
 \vecfield_{\minvar}(\curr[\minstate],\curr[\maxstate]) + \noise_{\minvar,\run} + \bias_{\minvar,\run}  \\
   \vecfield_{\maxvar}( {\curr[\minstate]},\curr[\maxstate]) + \noise_{\maxvar,\run} + \bias'_{\maxvar,\run}
\end{bmatrix}   }
\end{align}
where $\bias'_{\maxvar,\run} = \bias_{\maxvar,\run} + \vecfield_{\maxvar}(\next[\minstate],\curr[\maxstate])-\vecfield_{\maxvar}(\curr[\minstate],\curr[\maxstate])$.
By \eqref{eq:hhhh} and \cref{prop:algorithms}, $\bias'_{\maxvar,\run}$ satisfies \eqref{eq:errorbounds}, so our proof is complete.
\hfill
\qedsymbol

\bibliographystyle{icml}
\bibliography{IEEEabrv,bibtex/Bibliography-PM,bibtex/Bibliography-YPH,bibtex/Bibliography-MRK}

\end{document}